\documentclass[11pt,twoside]{amsart}

\numberwithin{equation}{section}
\usepackage{latexsym,amssymb,amsmath,amsthm,amsfonts,amscd}
\usepackage{enumerate}
\usepackage[table]{xcolor}
\usepackage{graphicx}
\usepackage{tikz}
\usepackage{multicol}
\usepackage{multirow}
\usepackage{tabularx}

\usetikzlibrary{matrix,arrows}

\definecolor{VerdeOlivo}{rgb}{0.3,0.5,0.1}
\definecolor{Magenta}{rgb}{.65,0.15,.2}
\definecolor{Gris}{gray}{0.3}

\newtheorem*{Theorem*}{Theorem} 
\newtheorem{Theorem}{Theorem}[section] 
\newtheorem{Definition}[Theorem]{Definition}
\newtheorem{Proposition}[Theorem]{Proposition}  
 
\newtheorem{Corollary}[Theorem]{Corollary}
\newtheorem{Remark}[Theorem]{Remark}
\newtheorem{Example}[Theorem]{Example}
    
\newtheorem{Conjecture}[Theorem]{Conjecture}

\theoremstyle{definition}
\newcommand{\bff}[1]{{\bf #1}}

\begin{document}


\title[Arithmetical structures on graphs]{Arithmetical structures on graphs}


\author{Hugo Corrales}
\email[H. ~Corrales]{hhcorrales@gmail.com}
\author{Carlos E. Valencia}
\email[C. ~Valencia\footnote{Corresponding author}]{cvalencia@math.cinvestav.edu.mx, cvalencia75@gmail.com}

\thanks{The second author was partially supported by SNI}
\address{
Departamento de
Matem\'aticas\\
Centro de Investigaci\'on y de Estudios Avanzados del IPN\\
Apartado Postal 14--740 \\
07000 Mexico City, D.F. 
} 

\keywords{arithmetical structures, $M$-matrices, edge subdivision, path, cycle, Catalan number, splitting vertices, merging vertices, delta-wye transformation, clique--star transformation.}
\subjclass[2010]{Primary 11D72, 15B48; Secondary 05C50, 05C76, 11D45, 11B83, 05E99, 11D68, 11C20.} 



\begin{abstract} 
Arithmetical structures on a graph were introduced by Lorenzini in~\cite{Lorenzini89} as some intersection matrices that arise in the study
of degenerating curves in algebraic geometry.
In this article we study these arithmetical structures, in particular we are interested in the arithmetical structures on complete graphs, paths, and cycles.  
We begin by looking at the arithmetical structures on a multidigraph from the general perspective of $M$-matrices.
As an application, we recover the result of Lorenzini about the finiteness of the number of arithmetical structures on a graph.
We give a description on the arithmetical structures on the graph obtained 
by merging and splitting a vertex of a graph in terms of its arithmetical structures.
On the other hand, we give a description of the arithmetical structures on the clique--star transform of a graph, 
which generalizes the subdivision of a graph.
As an application of this result we obtain an explicit description of all the arithmetical structures on the paths and cycles
and we show that the number of the arithmetical structures on a path is a Catalan number.
\end{abstract}

\maketitle


\section{Introduction}\label{intro}
Given a loopless multidigraph $G=(V,E)$, its \emph{generalized Laplacian matrix} is given by
\[
L(G,X_G)_{u,v}=\begin{cases}
- m_{u,v}&\textrm{if }u\neq v,\\
x_u&\textrm{if }u=v,
\end{cases}
\]
where $m_{u,v}$ is the number of arcs between $u$ and $v$ and $X_G=\{x_v|v\in V\}$ is a finite set of variables indexed by the vertices of $G$.
This generalized Laplacian matrix was introduced in~\cite{critical} and is very similar to the corresponding concept introduced by Godsil and Royle in~\cite[Section 13.9]{godsil}.
For any ${\bf d}\in \mathbb{Z}^{V}$, let $L(G,{\bf d})$ be the integral matrix that results by making $x_u={\bf  d}_u$ on $L(G,X_G)$,
which is called the Laplacian matrix of $G$ and $(\bff{d},\bff{r})$. 
Clearly, the \emph{adjacency matrix} of $G$ is equal to $-L(G,{\bf 0})$ 
and their \emph{Laplacian matrix} is equal to $L(G,{\bf deg}_G)$, where ${\bf deg}_G$ is the out degree vector of $G$. 
The Laplacian matrix of a graph is very important in spectral graph theory and in general in algebraic graph theory, 
see for instance~\cite{godsil} and the references therein.
The Laplacian matrices with which we work here essentially correspond to integral matrices with non-positive entries off the diagonal.

Some combinatorial properties of a multidigraph $G$ are coded in its Laplacian matrix.
For instance, $G$ is said to be \emph{strongly connected} if for any two vertices $u,v\in V$ there exists a directed path from $u$ to $v$.
It is well known that $G$ is strongly connected if and only if $L(G,{\bf deg}_G)$ is an irreducible matrix, see for instance~\cite{godsil}.
Recall that a square matrix $A$ is called reducible if there exists a permutation matrix $P$ such that
\[
PAP^t=\left(\begin{array}{cc}
A_1&*\\
0&A_2
\end{array}\right)
\]
for some square matrices $A_1$ and $A_2$. 
Moreover, is not difficult to check that $G$ is strongly connected if and only if $L(G,\bff{d})$ is an irreducible matrix for any vector $\bff{d}$ in $\mathbb{Z}^{V}$.

Now, we define the main object of this article.
An \emph{arithmetical graph} is a triplet $(G,{\bf d},\bff{r})$ given by a multidigraph $G$ and a pair of vectors 
$({\bf d},\bff{r})\in \mathbb{N}_+^V\times \mathbb{N}_+^V$ such that $\mathrm{gcd}(\bff{r}_v\, | \,v\in V(G))=1$ and
\[
L(G,{\bf d})\bff{r}^t=\bff{0}^t.
\]
Note that we impose the condition that all the entries of $\bff{d}$ and $\bff{r}$ are positive. 
Given an arithmetical graph $(G,{\bf d},\bff{r})$, we say that the pair $({\bf d},\bff{r})$ is an \emph{arithmetical structure} on $G$.
This concept was introduced by Lorenzini in~\cite{Lorenzini89} as some intersection matrices that arise in 
the study of degenerating curves in algebraic geometry, see for instance~\cite{components} for a geometric point of view.
Any simple graph $H$ has an arithmetical structure, given by $(\bff{d},\bff{r})=(\bf{deg}_H,\bff{1})$, 
which we call the Laplacian arithmetical structure on $H$.
For the rest of this article, unless otherwise specified, it will be understood that a graph is a multidigraph.

One of the main results given in~\cite{Lorenzini89} is that the number of arithmetical structures on a simple connected graph is finite.

\begin{Theorem*}[Lemma 1.6 \cite{Lorenzini89}]
{\it There exist only finitely many arithmetical structures on any connected simple graph.} 
\end{Theorem*}

Since the number of arithmetical structures on a graph is finite, is natural to ask about its possible descriptions. 
Consider
\[
\mathcal{A}(G)=\{(\bff{d},\bff{r})\in\mathbb{N}_+^{V(G)}\times \mathbb{N}_+^{V(G)} \,| \,(\bff{d},\bff{r})\textrm{ be an arithmetical structure on } G\}.
\]
Given $(\bff{d},\bff{r})\in \mathcal{A}(G)$, let
\[
K(G,\bff{d},\bff{r})=\mathrm{ker}(\bff{r}^t)/\mathrm{Im} \, L(G,\bff{d})^t
\]
be the critical group of $(G,\bff{d},\bff{r})$, which generalizes the concept of critical group of $G$ introduced in~\cite{biggs97}. 
In a similar way as for the critical group of $G$ (the cokernel of $L(G)$), this definition of critical group 
is closely related to the critical ideals (the determinantal ideals associated to the generalized Laplacian matrix) of $G$, 
see~\cite{critical} for a precise definition of critical ideals. 
Even more, by~\cite[Propositions $3.6$ and $3.7$]{critical}, we can recover the critical group of $(G,\bff{d},\bff{r})$ as an evaluation of the critical ideals of $G$.
Additionally, in general, given an integer matrix $M$, its critical group $K(M)$ is defined as the torsion part of its cokernel.

\medskip

The main objective of this paper is to study the arithmetical structures on a multidigraph. 
Firstly, in Section~\ref{m-matrices}, we connect the Laplacian matrix obtained from an arithmetical structure of a graph with $M$-matrices.
The subject of $M$-matrices is of considerable interest in mathematics with applications to numerical analysis, 
probability, economics, operations research, etc., see~\cite{Berman} and the references therein.
More precisely, we introduce a new class of $M$-matrices, called the almost non-singular $M$-matrices 
(\emph{all its proper principal minors are positive and its determinant is non-negative}), and we prove that any
Laplacian matrix obtained from an arithmetical structure belongs to this class.
After that, given a non-negative integral $n\times n$ matrix $B$ and $\alpha\geq 0$, we introduce the set
\[
\mathcal{A}_{\alpha}(B)=\{{\bf d}\in\mathbb{N}_+^n\, | \, A=\rm{diag}({\bf d})-B  \text{ is an } M\text{-matrix and } \rm{det}(A)= \alpha\}
\]
which is a generalization of the set of arithmetical structures on a graph.
The main result of Section~\ref{m-matrices} 
proves that $\mathcal{A}_{\alpha}(B)$ is finite for any $\alpha>0$.

{\bf Theorem~\ref{finitepositive}}
{\it If $B$ is a non-negative integral matrix, then $\mathcal{A}_{\alpha}(B)$ is finite for any $\alpha>0$.}

In Section~\ref{arithmetical} we present the relation between $M$-matrices and arithmetical graphs.
We give some basic properties of the Laplacian matrix associated to an arithmetical graph
and we characterize when a $Z$-matrix (a matrix with all its off-diagonal entries non-positive) is an irreducible almost non-singular $M$-matrix.

{\bf Theorem~\ref{almostr}}
{\it Let $M$ be a $Z$-matrix.
If there exists $\bff{r}> \bff{0}$ such that $M\bff{r}^t=\bff{0}^t$, then $M$ is an $M$-matrix.
Moreover, $M$ is an almost non-singular $M$-matrix with $\mathrm{det}(M)=0$ 
if and only if $M$ is irreducible and there exists $\bff{r}> \bff{0}$ such that $M\bff{r}^t=\bff{0}^t$.}

Using this result we generalize the finiteness result of Lorenzini. 

{\bf Theorem~\ref{finitezero}}
{\it If $M$ is a non-negative matrix with all the diagonal entries equal to zero, then $\mathcal{A}(M)\neq\emptyset$.
Furthermore, $\mathcal{A}(M)$ is finite if and only if $M$ is irreducible.}

We finish this section by given some arithmetical structures on the cone of a graph.

In Section~\ref{mergesplit} we present a way to construct arithmetical structures for the graphs obtained by merging and splitting vertices.
Given a graph $G$ and vertices $u$ and $u'$, let $m(G,u,u')$ be the graph obtained by merging the vertices $u$ and $u'$ into a new vertex $w$.

{\bf Theorem~\ref{merge}}
{\it If $(G,\bff{d},\bff{r})$ is an arithmetical graph such that $\bff{r}_u=\bff{r}_{u'}$ for some $u,u'\in V(G)$, then
$m(\bff{d},\bff{r})=(m(\bff{d}),m(\bff{r}))$ is an arithmetical structure of $m(G,u,u')$, where
\[
m(\bff{d})_v=
\begin{cases}
\bff{d}_u+\bff{d}_{u'} & \text{ if } v=w,\\
\bff{d}_v & \text{ otherwise}
\end{cases}
\]
and $m(\bff{r})\in \mathbb{N}^{V(m(G,u,u'))}$ is given by $m(\bff{r})_v=\bff{r}_v$ for all $v\in V(m(G,u,u'))$ and $w$ is the vertex obtained by merging the vertices $u$ and $u'$.}

Given a graph $G$, a vertex $u$ of $G$, and $A\subsetneq N_{G}(u)$, let $s(G,u,A)$ 
be the graph obtained by splitting the vertex $u$ into two vertices $w$ and $w'$ with neighborhoods $A$ and $N_G(u)-A$.

{\bf Theorem~\ref{split}}
{\it If $(G,\bff{d},\bff{r})$ is an arithmetical graph such that
\[
\bff{r}_u\big| \sum_{a\in A} \bff{r}_a \text{ for some } A\subsetneq N_{G}(u) \text{ and } u\in V(G),
\]
then $s(\bff{d},\bff{r})=(s(\bff{d}),s(\bff{r}))$ is an arithmetical structure of $s(G,u,A)$, where
\[
s(\bff{d})_v=
\begin{cases}
\frac{\sum_{a\in A} \bff{r}_a}{\bff{r}_u} & \text{ if } v=w,\\
\frac{\sum_{a\in N_{G}(u)-A} \bff{r}_a}{\bff{r}_u} & \text{ if } v=w',\\
\bff{d}_v & \text{ otherwise}
\end{cases}
\]
and $\bff{r}\in \mathbb{N}^{V(s(G,u,A))}$ is given by 
\[
s(\bff{r})_v=
\begin{cases}
\bff{r}_u & \text{ if } v=w,w',\\
\bff{r}_v & \text{ otherwise.}
\end{cases}
\]
}

In Section~\ref{sclique--star} the arithmetical structures on the clique--star transform of a graph are studied.
The clique--star transform $cs(G,C)$ of $G$ takes a clique of $G$ and replace it by a star with a new vertex as center.
The clique--star transformation generalizes the subdivision of an edge, adding pendant edges and the $\Delta-Y$ transformation on graphs.
We establish a relation between the arithmetical structures on $G$ and $cs(G,C)$.

{\bf Theorem~\ref{clique--star}}
{\it Let $G$ be a graph, $C$ be a clique of $G$, $(\bff{d},\bff{r})$ an arithmetical structure of $G$, and $\tilde{G}=cs(G,C)$.
If $v$ is an arbitrary vertex of $G$ and $\mathcal{A}'(\tilde{G})$ is the set of arithmetical structures $(\bff{d}',\bff{r}')$ of $\tilde{G}$ with $\bff{d}'_v=1$, then
\[
\mathcal{A}'(\tilde{G})=\{(\tilde{\bff{d}},\tilde{\bff{r}}) \,|\, (\bff{d},\bff{r})\in \mathcal{A}(G)\},
\]
where
\begin{equation*}
\begin{array}{ccc}
\tilde{\bff{d}}_u=cs(\bff{d},C)_u=
\begin{cases}
{\bf d}_u&\textrm{ if } u\not\in C,\\
{\bf d}_u+1&\textrm{ if } u\in C,\\
1&\textrm{ if }u=v,
\end{cases}
&\quad\textrm{and}\quad\quad&
\tilde{\bff{r}}_u=cs(\bff{r},C)_u=
\begin{cases}
\bff{r}_u&\textrm{ if } u\in V\\
\\
\displaystyle{\sum_{u\in C} \bff{r}_u}&\textrm{ if }u=v.
\end{cases}
\end{array}
\end{equation*}
}
Moreover, it can be proved that the critical group associated to any arithmetical structure $(\bff{d},\bff{r})$ of $G$ 
and the critical group associated to the arithmetical structure obtained from $(\bff{d},\bff{r})$ by applying a clique--star transformation are isomorphic.

Finally, in Section~\ref{app} we apply the result obtained in Section~\ref{sclique--star} 
to give explicit descriptions of the arithmetical structures on paths and cycles.

{\bf Theorem~\ref{DPnrecursive}}
{\it If $(\bff{d},\bff{r})$ is an arithmetical structure of $P_n$ for $n\geq 3$ and ${\bf d}\neq (1,2,\ldots,2,1)$,
then there exists a non-terminal vertex $v$ of $P_n$ such that 
\[
{\bf d}_v=1 \text{ and }{\bf d}_u>1\text{ for all }u\in N_{P_n}(v).
\]
Moreover, any arithmetical structure $(\bff{d},\bff{r})$ of $P_n$ different from the Laplacian one
can be obtained from an arithmetical structure of $P_{n-1}$ by subdividing an edge.
}

The number of arithmetical structure of $P_n$ can be calculated.

{\bf Theorem~\ref{catalan}}
{\it The number of arithmetical structures on the path $P_{n+1}$ is equal to the Catalan number 
\[
C_{n}=\frac{1}{n+1}\binom{2n}{n}.
\]
}

{\bf Theorem~\ref{DCnrecursive}}
{\it If $(\bff{d},\bff{r})$ is an arithmetical structure of $C_n$ for $n\geq 4$ and ${\bf d}\neq 2\cdot {\bf 1}$,
then there exists a vertex $v$ of $C_n$ such that 
\[
{\bf d}_v=1 \text{ and }{\bf d}_u>1\text{ for all }u\in N_{C_n}(v).
\]
Moreover, any arithmetical structure of $C_n$ different from $(\bff{2},\bff{1})$ 
can be obtained from an arithmetical structure of $C_{n-1}$ by subdividing an edge.
}

We conclude by posing a conjecture about the number of the arithmetical structures on a graph.

{\bf Conjecture~\ref{conj}}
If $G$ is a simple connected graph with $n$ vertices, then 
\[
|\mathcal{A}(P_n)| \leq |\mathcal{A}(G)|\leq |\mathcal{A}(K_n)|.
\]

Throughout this paper we use the following partial order over $\mathbb{R}^V$.
If $\mathbf{d},\mathbf{e}\in\mathbb{R}^V$, then we say that $\mathbf{d}\leq \mathbf{e}$ if and only if $\mathbf{d}_v\leq \mathbf{e}_v\ \forall\,v\in V$.
It is well known that $\leq$ is a well partial order over $\mathbb{N}_+^V$, which means that every 
infinite sequence of elements on $\mathbb{N}_+^V$ contains an increasing pair.
An important equivalent property, know as Dickson's Lemma (see~\cite{dickson}), is that for any $S\subseteq \mathbb{N}_+^V$ 
the set $\textrm{min}(S)=\left\{ \mathbf{x}\in S\,\big|\,\mathbf{y}\not\leq \mathbf{x}\ \forall\, \mathbf{y}\in S\right\}$ is finite.
For $\mathbf{d},\mathbf{e}\in\mathbb{N}_+^V$ we say that $\mathbf{d}<\mathbf{e}$ if and only if $\mathbf{d}\leq \mathbf{e}$ and $\mathbf{d}\neq \mathbf{e}$.
Is important to note that $\mathbf{d}<\mathbf{e}$ does not mean that $\mathbf{d}_v< \mathbf{e}_v$ for all $v\in V$.


\section{$M$-matrices}\label{m-matrices}
In this section we recall the classical concept of an $M$-matrix and we introduce a new class of $M$-matrices 
whose proper principal minors are positive and its determinant is non-negative, which will be called the almost non-singular $M$-matrices. 
After that we prove that
\[
\mathcal{A}_{\alpha}(M)=\{{\bf d}\in\mathbb{N}_+^n\, | \, A=\rm{diag}({\bf d})-M  \text{ is an } M\text{-matrix and } \rm{det}(A)= \alpha\}
\]
is finite for all $\alpha> 0$ and any non-negative integral $n\times n$ matrix $M$ with all the diagonal entries equal to zero, see Theorem~\ref{finitepositive}.
In the following, \emph{matrix} always means \emph{square matrix}. 
Recall that a real matrix is called non-negative if all its entries are non-negative real numbers.

A real matrix $A=(a_{i,j}) \in \mathbb{R}^{n\times n}$ is called a $Z$-matrix if $a_{i,j}\leq 0$ for all $i\neq j$.
The spectral radius $\rho(M)$ of a matrix $M$ is defined by 
\[
\rho(M)=\rm{max}\{|\lambda| | \lambda \in \sigma(M)\},
\] 
where $\sigma(M)$ is the spectrum of $M$, that is, the set of complex eigenvalues of $M$.

\begin{Definition}
A $Z$-matrix $A$ such that
\[
A = \alpha I- M,
\]
for some non-negative matrix $M$ with $\alpha \geq \rho(M)$ is called an $M$-matrix. 
\end{Definition}

The study of $M$-matrices can be divided into two major parts: 
non-singular $M$-matrices (see \cite[Section 6.2]{Berman}) and singular $M$-matrices (see \cite[Section 6.4]{Berman}).
A $Z$-matrix is an $M$-matrix if and only if all its principal minors are non-negative, 
and is a non-singular $M$-matrix if and only if all its principal minors are positive, see~\cite[Chapter 6]{Berman}.
$M$-matrices are very important in a broad range of mathematical disciplines. 
The book~\cite{Berman} by Berman and Plemmons study non-singular and singular $M$-matrices.
Recently, $M$-matrices have been studied in the context of chip-firing games, see~\cite{klivans} and the references contained there.

In this paper we restrict our attention to the following subclass of singular $M$-matrices.

\begin{Definition}
A real matrix $A=(a_{i,j})$ is called an \emph{almost non-singular $M$-matrix} 
if $A$ is a $Z$-matrix, all its proper principal minors are positive and its determinant is non-negative.
\end{Definition}

This definition is motivated by the fact that the Laplacian matrix of a connected arithmetical 
graph is an almost non-singular matrix, see Corollary~\ref{arithmeticalalmost}.
A crucial fact about an arithmetical graph is that its asscociated Laplacian matrix is singular of maximal rank.
In this sense, the class of almost non-singular $M$-matrices is in between the class of singular $M$-matrices and the non-singular $M$-matrices.
For example, a singular irreducible $M$-matrix is an almost non-singular $M$-matrix, see~\cite[Theorem 6.4.16, page 156]{Berman}.

The class of $M$-matrices admit many equivalent definitions, in fact \cite{Berman} lists more than $80$
ways to characterize $M$-matrices as monotone operators on $\mathbb{R}_+^n$.
The class of almost non-singular $M$-matrices admit the following characterization, which will play a central role in the sequel.

\begin{Theorem}\label{almost}
If $M$ is a real $Z$-matrix, then the following conditions are equivalent:
\begin{itemize}
\item[(1)] $M$ is an almost non-singular $M$-matrix.
\item[(2)] $M+D$ is a non-singular $M$-matrix for any diagonal matrix $D\gneq 0$.
\item[(3)] $\det(M)\geq0$ and 
$\textrm{det}(M+D)\gneq \textrm{det}(M+D')\gneq 0$ for all diagonal matrices $D\gneq D'\gneq 0$.
\end{itemize}
\end{Theorem}
\begin{proof}
(1)$\Rightarrow$(2)
Given $1\leq s \leq n$, let $E_s=(e_{ij})$ be the matrix with $e_{ij}$ equal to $1$ for $i,j=s$ and $0$ otherwise.
First, we will prove that $M'=M+d\cdot E_s$ is a non-singular $M$-matrix for any $d>0$.
That is, we need to prove that all the principal minors of $M'$ are positive.
Let $\emptyset \neq I\subseteq [n]$.
If $s\not\in I$, then $M'[I;I]=M[I;I]>0$.
The last inequality is because $M$ is an almost non-singular $M$-matrix and $I\neq [n]$.
On the other hand, if $s\in I$, then 
\[
\det(M'[I;I])=\det(M[I;I])+d\cdot \det(M[I\setminus s;I\setminus s])\overset{(I\setminus s\neq [n],\, d>0)}{>}\det(M[I;I])\geq 0.
\]

Finally, since any diagonal matrix $D$ is equal to $\sum_{i=1}^n d_{i}\cdot E_{i}$ for some $d_i\in \mathbb{R}_+$, 
then the result follows by using the previous case several times. 

\medskip

(2)$\Rightarrow$(3)
We first prove that $\det(M)\geq0$.
If we take $D_m=(1/m)I_n$ ($m\in\mathbb{N}_+$), then $\det(M+D_m)>0$ and so $\lim_{m\rightarrow\infty} \det(M+D_m)\geq 0$.
Since $\lim_{m\rightarrow\infty}  D_m=0$ and the determinant is a continuous function with respect to the Hilbert--Schmidt norm, $\det(M)\geq 0$.

Now, let $D\gneq D'\gneq 0$ be diagonal matrices.
By hypothesis, $M+D'$ is a non-singular $M$-matrix and in particular an almost non-singular $M$-matrix.
Following the arguments and notation used to prove (1)$\Rightarrow$(2) we get that $M+D'+E_s$ 
is an almost non-singular $M$-matrix for any $1\leq s\leq n$ and $\det(M+D'+E_s)>\det(M+D')$.
In a similar way, it is not difficult to prove that $\det(M+D'+F)>\det(M+D')$ for any diagonal matrix $F>0$.
Now, clearly the result follows by taking $F=D-D'$.

On the other hand, let $F=D-D'\gneq 0$, let $f_{ii}$ be the first non-zero diagonal entry of $F$,  let $C'=M+D'$, and let $C=C'+F$.
Then, since $C'$ is a non-singular $M$-matrix and $f_{ii}>0$, $\det(C)=\det(C')+f_{ii}\cdot \det(C\big[[n]\setminus i,[n]\setminus i\big])>\det(C')$.

\medskip

(3)$\Rightarrow$(1)
Let $f:\mathbb{R}^n\rightarrow\mathbb{R}$ be given by $f(x_1,\ldots,x_n)=\det(M+\textrm{diag}(x_1,\ldots,x_n))$.
By hypothesis $f$ is a non-negative and increasing function on $(\mathbb{R}_+\cup \{0\})^n$. 
Also it is not difficult to see that 
\[
f(x_1,\ldots,x_n)=\sum_{I\subseteq [n]} \det(M[I;I])x_{I^c},
\]
where $x_{J}=\prod_{j\in J} x_j$ for all $J\subseteq [n]$.

First we prove that $M$ is an $M$-matrix.
By~\cite[Theorem 6.4.6]{Berman}, we only need to prove that $\det(M[J;J])\geq 0$ for each $J\subseteq [n]$.
Let $J\subseteq [n]$.
If $J=[n]$, then $M[J;J]=M$ and thus $\det(M[J;J])=f(0,\ldots,0)\geq 0$.
If $J=[n]\setminus j$ for some $j\in [n]$, then $M[J;J]=\partial f/\partial x_j(0,\ldots,0)>0$ since 
$\partial f/\partial x_j$ is positive on $(\mathbb{R}_+\cup \{0\})^n$.

If $J\subsetneq[n]\setminus j$ for some $j\in [n]$, then let $a_i=x$ for $i\not\in J$ and $a_i=0$ for $i\in J$.
Thus, if $\det(M[J;J])<0$, then the leading coefficient of $\partial f/\partial x(a_1,\ldots,a_n)$ will be $\det(M[J;J])$, 
which is a contradiction since $\partial f/\partial x_i$ is positive on $(\mathbb{R}_+\cup \{0\})^n$.
Thus,  $\det(M[J;J])\geq 0$.

Since we already proved that $\det(M)\geq0$ and that $\det(M[J;J])>0$ if $J\subseteq [n]$ with $|J|=n-1$, then,  
in order to prove (1), we need to show that $\det(M[J;J])>0$ for each $J\subseteq [n]$ with $|J|<n-1$.
Let $J\subseteq [n]$ with $|J|<n-1$.
Since $|J|<n-1$, there exists $j\in [n]$ such that $J\subsetneq [n]\setminus j$.
Let $I=[n]\setminus j$.
Since $M$ is an $M$-matrix, it follows that $M[I;I]$ is also an $M$-matrix.
But $\det(M[I;I])>0$ since  $|I|=n-1$.
This means that $M$ is a non-singular $M$-matrix.
By~\cite[Theorem 6.2.3]{Berman} all the principal minors of $M[I;I]$ are positive.
In particular, $\det(M[J;J])>0$.
\end{proof}

Using any algebraic software package, such as {\it Sage, Macaulay, Maple} or {\it Mathematica}, it is not difficult to check when a matrix is an $M$-matrix. 
\begin{Remark}
Let $M$ be a real $Z$-matrix and 
\[
f_M(\bff{x})=\det(M+\textrm{diag}(x_1,\ldots,x_n))\in \mathbb{R}[x_1,\ldots,x_n].
\] 
Then $M$ is an $M$-matrix (non-singular $M$- matrix) if and only if the coefficients of the polynomial $f_M$ are non-negative (positive).
In a similar way, $M$ is an almost non-singular $M$-matrix if and only if all the coefficients, 
except maybe the constant term, of the polynomial $f_M$, are positive.

For example, if
\[
M=\left(\begin{array}{ccc}
1& -1& 0\\
0& 1& -1\\
-1&-1&2
\end{array}\right),
\]
then $f_M(\bff{x})=x_1x_2x_3+2x_1x_2+x_1x_3+x_2x_3+x_1+2x_2+x_3$.
Thus, $M$ is an almost non-singular matrix $M$-matrix, but not a non-singular $M$-matrix.
\end{Remark}

In this article we are primarily interested in $M$-matrices with integral entries.
Some ideals associated to integral matrices
(called matrix ideals, which includes Laplacian ideals and lattice ideals as toppling ideals) of several families of matrices like 
Pure Binomial (PB), Critical Binomial (CB), Generalized Critical Binomial (GCB) 
and its variants where all the entries of the matrix are positive (PPB, PCB and GPCB) are studied in~\cite{Ocarroll}.
In our context, PB, PPB, GPCB and PCB matrices correspond to $Z$-matrices, adjacency matrices, Laplacian matrices associated to arithmetical structures 
and Laplacian matrices associated to the Laplacian arithmetical structure of the complete graph, respectively.
Moreover, irreducible GCB matrices are almost non-singular $M$-matrices with $\mathrm{det}(M)=0$, see Theorem~\ref{almostr}.

%

Given $\alpha\geq 0$ and a non-negative integral $n\times n$ matrix $M$ with all the diagonal entries equal to zero, let
\[
\mathcal{A}_{\geq \alpha}(M)=\{{\bf d}\in\mathbb{N}_+^n\, | \, A=\rm{diag}({\bf d})-M  \text{ is an } M\text{-matrix and } \rm{det}(A)\geq \alpha\}.
\]
Also, let $\mathcal{A}_{\alpha}(M)=\{{\bf d}\in\mathcal{A}_{\geq \alpha}(M) \, |\, \rm{det}(\rm{diag}({\bf d})-M )= \alpha\}$.
This set is closely related to the set of arithmetical structures on a graph.
More precisely, it is related to the case when $M$ is equal to the adjacency matrix of $G$ and $\alpha=0$.
However, in order to recover the main properties of the arithmetical structures on a graph 
we need to add some extra conditions in order to get the right definition, see Definition~\ref{defas}.
If $M$ is an almost non-singular $M$-matrix, then by Theorem~\ref{almost} we have that
\[
\mathcal{A}_{\geq \alpha}(M)=\textrm{min}\, \mathcal{A}_{\geq \alpha}(M)+(\mathbb{N}_+\cup \{0\})^n,
\]
where $\textrm{min}\,\mathcal{A}_{\geq\alpha}(M)=\{{\bf d}\in\mathcal{A}_{\geq\alpha}(M)\, | \, \text{ if } {\bf d}'\leq \bff{d} \text{ for some } \bff{d}'\in\mathcal{A}_{\geq\alpha}(M), \text{ then } \bff{d}'=\bff{d} \}$.
That is, $\mathcal{A}_{\geq \alpha}(M)$ is a monoid and infinite.
However, as the following theorem shows, the set $\mathcal{A}_{\alpha}(M)$ is finite when $M$ is a non-negative integral matrix and $\alpha>0$.
\begin{Theorem}\label{finitepositive}
If $M$ is a non-negative integral matrix, then $\mathcal{A}_{\alpha}(M)$ is finite for any $\alpha>0$.
\end{Theorem}
\begin{proof}
We claim that $\mathcal{A}_\alpha(M)\subseteq \textrm{min}\,\mathcal{A}_{\geq\alpha}(M)$.
We prove this by contradiction. 
Let $\mathbf{d}\in\mathcal{A}_\alpha(M)$ and assume that $\mathbf{d}\not\in\textrm{min}\,\mathcal{A}_{\geq\alpha}(M)$.
This means that there exists an $\mathbf{e}\in\mathcal{A}_{\geq\alpha}(M)$ such that $\mathbf{e}\lneq \mathbf{d}$.
Since $\textrm{det}(\rm{diag}(\mathbf{e})-M)\geq\alpha>0$, $\rm{diag}(\mathbf{e})-M$ is a non-singular $M$-matrix.
By Theorem~\ref{almost}, $\textrm{det}(\rm{diag}(\mathbf{d})-M)>\textrm{det}(\rm{diag}(\mathbf{e})-M)\geq\alpha$,
which is a contradiction since $\textrm{det}(\rm{diag}(\mathbf{d})-M)=\alpha$.

Finally, since $\mathcal{A}_{\geq\alpha}(M)\subseteq\mathbb{N}_+^n$, by Dickson's lemma 
$\textrm{min}\,\mathcal{A}_{\geq\alpha}(M)$ is finite and thus $\mathcal{A}_\alpha(M)$ is also finite.
\end{proof}

The inclusion $\mathcal{A}_\alpha(M)\subseteq \textrm{min}\,\mathcal{A}_{\geq\alpha}(M)$ in general is not an equality.

\begin{Example}\label{exampleinfinite}
If 
\[
M=\left(\begin{array}{ccc}0&1&1\\1&0&1\\0&1&0\end{array}\right),
\] 
then $\mathcal{A}_{6}(M)=\{ (3,2,2)^t,(2,2,3)^t\}$ and $\mathrm{min}\,\mathcal{A}_{\geq 6}(M)=\{ (3,2,2)^t,(2,3,2)^t,(2,2,3)^t\}$.
\end{Example}

The special case of $\mathcal{A} _{\alpha}(M)$ when $\alpha$ is equal to zero is more difficult to treat.
For instance, if $M$ is reducible, then $\mathcal{A} _{0}(M)$ can be infinite, as the next example shows.
In the next section we deal with this special case, see Theorem~\ref{finitezero}.
\begin{Example}\label{Exampleinfinite}
Let
\[
M=\left(\begin{array}{cccc}0&0&1&0\\1&0&1&1\\1&0&0&0\\1&1&1&0\end{array}\right).
\]
It is not difficult to check that $\left\{(1,x,1,y)^t|x,y\in\mathbb{N}_+\right\}\subsetneq \mathcal{A}_0(M)$.
That is, $\mathcal{A}_0(M)$ is infinite.  
On the other hand, since
\[
\left(\begin{array}{cccc}0&0&0&1\\0&1&0&0\\0&0&1&0\\1&0&0&0\end{array}\right)
\left(\begin{array}{cccc}0&0&1&0\\1&0&1&1\\1&0&0&0\\1&1&1&0\end{array}\right)
\left(\begin{array}{cccc}0&0&0&1\\0&1&0&0\\0&0&1&0\\1&0&0&0\end{array}\right)^t
=
\left(\begin{array}{cccc}0&1&1&1\\1&0&1&1\\0&0&0&1\\0&0&1&0\end{array}\right),
\]
then $M$ is reducible.
\end{Example}


\section{arithmetical graphs and almost non-singular $M$-matrices}\label{arithmetical}
In this section we show that the concept of an almost non-singular $M$-matrix generalizes the concept of an arithmetical graph.
We begin by giving the definition of an arithmetical graph. 
A key result is the characterization of when an irreducible $Z$-matrix is an almost non-singular $M$-matrix.
More precisely, an irreducible matrix $M$ is an almost non-singular $M$-matrix with $\mathrm{det}(M)=0$ 
if and only if there exist $\bff{r}> \bff{0}$ such that $M\bff{r}^t=0$ (see Theorem~\ref{almostr}). 
Using this result, we obtain that $\mathcal{A}(M)$ is finite if and only if $M$ is irreducible.
This allows characterizing the multidigraphs that have a finite number of arithmetical structures, see Corollary~\ref{finitearithmetical}.

Given a graph $G=(V,E)$, a pair $({\bf d},\bff{r})\in \mathbb{N}_+^V\times \mathbb{N}_+^V$ such that $\textrm{gcd}(\bff{r}_v \, | \, v\in V)=1$ and
\[
L(G,{\bf d})\bff{r}^t=\bff{0}^t,
\]
where $L(G,{\bf d})=\mathrm{diag}(\bff{d})-A(G)$, is called an \emph{arithmetical structure} of $G$.
Note that we require that all the entries of $\bff{d}$ and $\bff{r}$ are necessarily positive integers.
The matrix $L(G,{\bf d})$ is the Laplacian matrix of the arithmetical graph $(G,{\bf d},\bff{r})$.
We say that the triple $(G,{\bf d},\bff{r})$ is an \emph{arithmetical graph}.
Any graph $G$ has a Laplacian arithmetical structure, given by $(\bff{d},\bff{r})=(\bf{deg}_G,\bff{1})$.

We can generalizes the concept of arithmetical graph to any integer irreducible 
$n\times n$ matrix $M$ such that there exists $\bff{r}\in\mathbb{N}_+^n$ with $M\bff{r}^t=\bff{0}^t$.


\begin{Definition}\label{defas}
Given a non-negative integer $n\times n$ matrix $M$ with all the diagonal entries equal to zero, let 
\[
\mathcal{A}(M)=\{(\bff{d},\bff{r})\in\mathbb{N}_+^{n}\times \mathbb{N}_+^{n} \,| \,
[\mathrm{diag}(\bff{d})-M]\bff{r}^t=\bff{0}^t \text{ and } gcd(\bff{r}_v \, | \, v\in V)=1\}.
\]
In a similar way, let $\mathcal{D}(M)=\{\bff{d}\,|\,(\bff{d},\bff{r})\in \mathcal{A}(M)\}\textrm{ and } \mathcal{R}(M)=\{\bff{r}\,|\,(\bff{d},\bff{r})\in \mathcal{A}(M)\}$.
\end{Definition}

Clearly $\mathcal{A}(G)=\mathcal{A}(A(G))$.
If $M$ is an integer non-negative matrix with diagonal entries not necessarily equal to zero, then $\mathcal{A}(M)=\mathcal{A}\left(M-\textrm{diag}(M)\right)+\textrm{diag}(M)$,
therefore it is plausible to assume, without loss of generality, that $M$ is a non-negative matrix with zero diagonal.

\begin{Theorem}\label{almostr}
Let $M$ be a $Z$-matrix.
If there exists $\bff{r}> \bff{0}$ such that $M\bff{r}^t=\bff{0}^t$, then $M$ is an $M$-matrix.
Moreover, $M$ is an almost non-singular $M$-matrix with $\mathrm{det}(M)=0$ 
if and only if $M$ is irreducible and there exists $\bff{r}> \bff{0}$ such that $M\bff{r}^t=\bff{0}^t$.
\end{Theorem}
\begin{proof}
Let $M'=M\mathrm{diag}(\bff{r})$. 
It suffices to show that $M'$ is an $M$-matrix.
Since $M'\bff{1}=0$, then we can assume without loss of generality that $\bff{r}=1$.
Now, let $S=M[I,I]$ with $I\subset [n]$.
For simplicity we can assume without loss of generality that $I=[s]$.
By the Gershgorin circle theorem, we have that each real eigenvalue of $S$ is contained in  
\[
\bigcup_{i=1}^s [m_{i,i}-t_i,m_{i,i}+t_i]
\] 
where $t_i=\sum_{j\in[s]\setminus i} |m_{i,j}|$.
Since $M\bff{1}=\bff{0}$ and $m_{i,j}\leq 0$ for all $i\neq j$, then
\[
t_i=\sum_{j\in [s]\setminus i} -m_{i,j}=m_{i,i}+\sum_{j=s+1}^n m_{i,j}\leq m_{i,i}>0.
\]
This means that all the real eigenvalues of $S$ are non-negative.
Moreover, since $\det(S)$ is the product of the real eigenvalues of $S$ and the norms of the non-conjugate complex ones, $\det(S)$ is non-negative.
That is, if $M$ is a $Z$-matrix, and such that there exists $\bff{r}> \bff{0}$ such that $M\bff{r}^t=0$, 
then $M$ is an $M$-matrix, see~\cite[Theorem $6.4.6\, (A_1)$, page 149]{Berman}.

$(\Leftarrow)$
For this part, it only remains to prove that $\det(S)>0$.
Assume that $\det(S)=0$.
Then there exists a non-zero ${\bf t}\in\mathbb{Z}^m$ such that $S{\bf t}={\bf 0}$.
Let $|{\bf t}_i|$ be the maximum among $|{\bf t}_1|,\ldots,|{\bf t}_s|$. 
Since ${\bf t}$ is non-zero, $|{\bf t}_i|>0$.
Taking ${\bf t}'=\frac{1}{{\bf t}_i}\bff{t}$, we can assume that ${\bf t}_i=1$ and that ${\bf t}_j\leq 1$ for each $j=1,\ldots,s$.
Then, $m_{i,1}{\bf t}_1+\cdots+m_{i,s}{\bf t}_s=0$, and since ${\bf t}_i=1$, we get 
\[
m_{i,i}=-m_{i,1}{\bf t}_1-\cdots -m_{i,i-1}{\bf t}_{i-1}-m_{i,i+1}{\bf t}_{i+1}-\cdots-m_{i,s}{\bf t}_s.
\]
But, since $M\bff{1}=\bff{0}$, we know that $m_{i,i}=-m_{i,1}-\cdots -m_{i,i-1}-m_{i,i+1}-\cdots-m_{i,n}$, and so
\[
m_{i,s+1}+\cdots+m_{i,n}=({\bf t}_1-1)m_{i,1}+\cdots +({\bf t}_{i-1}-1)m_{i,i-1}+({\bf t}_{i+1}-1)m_{i,i+1}+\cdots+({\bf t}_s-1)m_{i,s}.
\]
Since $({\bf t}_j-1)\leq0$ for each $j=1,\ldots,s$ and $m_{i,j}\leq 0$ for all $j\neq i$, 
in the above expression the right hand side is non-negative and the left hand side is non-positive.
Thus, $m_{i,j}=0$ for all $j\neq i$, which is a contradiction to the assumption that $M$ is irreducible.
Note that this part there follows~\cite[Theorem 6.4.16, page 156]{Berman}. 

$(\Rightarrow)$
First, since $\mathrm{det}(M)=0$, clearly there exists $0\neq \bff{r}\in \mathbb{R}$ such that $M\bff{r}^t=0$.
But we need to prove that there exists an $\bff{r}$ with all its entries positive. 
Let assume that this is not true.
Since the properties of being a $Z$-matrix, of being an almost non-singular $M$-matrix, 
and haing  $\mathrm{det}(M)=0$, are invariant under similarity via a permutation matrix ($M'=PMP^t$ with $P$ a permutation matrix),
we can assume without loss of generality that the first $0<s<n$ entries of $\bff{r}$ are less than or equal to zero.
Note that if $s=n$, then $\bf{r}<0$ and $M(-\bff{r})^t=0$ with $-\bff{r}>0$.
Let $N=M[[s],[n]]$ and decompose $N$ as $[S,T]$ where $S=M[[s],[s]]$, that is, $N$ is the matrix obtained from the first $s$ rows of $M$
and $S$ is formed by the first $s$ rows and columns of $M$.
Also, let $\bff{r}=(\bff{r_1},\bff{r}_2)$ where $\bff{r}_1$ is the vector with the first $s$ entries of $\bff{r}$.
Since $M\bff{r}^t=0$, then $N\bff{r}=0$, that is, $S\bff{r}_1=-T\bff{r}_2$.
Now, since $M$ is a $Z$-matrix, $-T\bff{r}_2\geq 0$.
Moreover, since $M$ is an almost non-singular $M$-matrix, 
$S$ is a non-singular $M$-matrix and by~\cite[Theorem 6.2.3 ($N_{39}$)]{Berman} and the fact that $\bff{r}_1\leq 0$, $\bff{r}_1=0$.
Therefore $T\bff{r}_2=0$ and since $\bff{r}\neq 0$, then $T=0$.
That is, $M$ is reducible.
Now, let $R=M[I,I]$ where $\emptyset\neq I=[s+1,\ldots,n]$, that is, $S$ and $R$ are principal blocks of $M$.
Since $T=0$, $\mathrm{det}(M)=\mathrm{det}(S)\mathrm{det}(R)=0$; this is a contradiction to the fact that $M$ is an almost non-singular $M$-matrix.
\end{proof}

A consequence of Theorem~\ref{almostr} is that $\{ \bff{d}\, | \, (\bff{d},\bff{r})\in \mathcal{A}(M)\}\subseteq \mathcal{A}_0(M)$.
Note that in Theorem~\ref{almostr} the $\bff{r}$ can always be chosen as a positive integer with $\gcd\{{\bf r'}_v | v\in V\}=1$. 
Here $\bff{r}$ plays a very important role and has an algebraic meaning.
More precisely, the existence of $\bff{r}$ means geometrically that all the rows of $M$ live in a hyperplane and
algebraically means that some of the ideals associated with $M$ are graded or homogeneous, see for instance~\cite{Ocarroll}.
Moreover, some of its associated ideals, such as its matrix ideal, are graded.
The fact that an ideal is homogeneous is very important in commutative algebra.
In some sense we can think that almost non-singular irreducible $M$-matrices with $\mathrm{det}(M)=0$ are the graded $M$-matrices.
An immediate consequence of Theorem~\ref{almostr} is that the Laplacian matrix of a strongly connected 
arithmetical graph is an irreducible almost non-singular $M$-matrix with $\mathrm{det}(L(G,\bff{d}))=0$, see Theorem~\ref{arithmeticalalmost}.
Theorem~\ref{almostr} can be compared with~\cite[Theorems 6.4, 7.5 and 7.6]{Ocarroll}.

Another immediate consequence of Theorem~\ref{almostr} is the following result (see~\cite[Theorems 6.4 and 7.5]{Ocarroll}):
\begin{Corollary}\label{transpose}
If $M$ is an irreducible $Z$-matrix, then there exists $\bff{r}> \bff{0}$ such that $M\bff{r}^t=0$
if and only if there exists $\bff{s}> \bff{0}$ such that $M^t\bff{s}^t=0$.
\end{Corollary}
\begin{proof}
This follows from the fact that $M$ is an irreducible almost non-singular $M$-matrix with $\mathrm{det}(M)=0$ 
if and only if $M^t$ is an irreducible almost non-singular $M$-matrix with $\mathrm{det}(M^t)=0$.
\end{proof}

Now we will present a way to compute the adjoint matrix of $M$ as a function of $\bff{r}$ and $\bff{s}$.
\begin{Proposition}\label{adjoint}
Let $M$ be a $Z$-matrix.
Then $M$ is an almost non-singular $M$-matrix with $\mathrm{det}(M)=0$ if and only if $\bff{r}> \bff{0}$ and
\[
\mathrm{Adj}(M)= |K(M)|\bff{r}^t\bff{s}>\bff{0}, 
\]
where $\ker_{\mathbb{Q}}(M)=\langle \bff{r}\rangle$ and $\ker_{\mathbb{Q}}(M^t)=\langle \bff{s}\rangle$.
\end{Proposition}
\begin{proof}
$(\Rightarrow)$
This follows from Theorem~\ref{almostr}, ~\cite[Proposition 2.1]{snf} and the fact that all the proper principal minors of $M$ are positive.

$(\Leftarrow)$
Assume that $M$ is reducible, that is, there exists a permutation matrix such that 
\[
PMP^t=\left(\begin{array}{cc}M_1&*\\0&M_2\end{array}\right).
\]
Since $M\bff{r}=\bff{0}$, then $\mathrm{det}(M_2)=0$ and therefore the adjoint of $M$ has at least one entry equal to zero, a contradiction.
Thus the result follows from Theorem~\ref{almostr}.
\end{proof}

Another possible characterization of an almost non-singular $M$-matrix is that its 
principal submatrices of maximal size ($M_{i,i}=M[i^c,i^c]$) are non-singular $M$-matrices.
The converse of this is contained in the following result.

\begin{Proposition}\label{extensionm}
If $M$ is a non-singular $M$-matrix, then there exists an irreducible almost non-singular $M$-matrix $M'$ with $\mathrm{det}(M')=0$ such that $M'_{1,1}=M$.
\end{Proposition}
\begin{proof}
Since $M$ is a non-singular $M$-matrix, there exists a real vector $\bff{r}>0$ such that $M\bff{r}> 0$, see for instance~\cite[page 136 ($I_{27}$)]{Berman}.
Moreover, since $M$ is integral, we can asume that $\bff{r}$ is integral.
Now, let $\bff{a}=M\bff{r}>0$, $\bff{r}'=(1,\bff{r})$ and
\[
M'=
\left(\begin{array}{cc}
\bff{r}M^t\bff{r}^t & -\bff{a}^t \\
-\bff{a} & M \\
\end{array}\right).
\]
It is not difficult to check that $M'\bff{r}'=0$.
Since the underlying graph of $M'$ is the cone of the underlying graph of $M$,
$M'$ is irreducible and the result follows by Theorem~\ref{almostr}.
\end{proof}

That is, any non-singular $M$-matrix can be extended to an almost non-singular irreducible $M$-matrix with $\mathrm{det}(M)=0$.

\begin{Corollary}\label{arithmeticalalmost}
If $(G,\bff{d},\bff{r})$ is a strongly connected arithmetical graph, then $L(G,\bff{d})$ 
is an almost non-singular $M$-matrix with $\mathrm{det}(L(G,\bff{d}))=0$.
\end{Corollary}
\begin{proof}
Let $M=L(G,{\bf d})=\textrm{Diag}(d)-\textrm{A}(G)$.
Since $G$ is a strongly connected graph if and only if $M$ is irreducible, then the result follows by applying Theorem~\ref{almostr}.
\end{proof}

If $G$ is a multidigraph, its adjacency matrix $A(G)$ is always a non-negative matrix with zeros on the diagonal.
On the other hand, if $M$ is a non negative matrix with zeros on the diagonal, then there exists a unique multidigraph $G_M$ such that $M=A(G_M)$.
The graph $G_M$ is called the underlying multidigraph of $M$.
The next theorem uses this correspondence to establish necessary and sufficient conditions 
on a non-negative matrix $M$ so that $\mathcal{A}(M)$ is finite.

\begin{Theorem}\label{finitezero}
If $M$ is a non-negative matrix with all diagonal entries equal to zero, then $\mathcal{A}(M)\neq\emptyset$.
Furthermore, $\mathcal{A}(M)$ is finite if and only if $M$ is irreducible.
\end{Theorem}
\begin{proof}
Let $G_M$ be such that $A(G_M)=M$. 
Note that $G_M$ is similar to the underlying graph of $M$, in some sense $G_M$ is a weighted version of the underlying graph of $M$.
Moreover $G_M$ is strongly connected if and only if the underlying graph of $M$ is strongly connected.
By Corollary~\ref{arithmeticalalmost} we only need to prove that $G_M$ has at least one arithmetical structure.
To see this, let ${\bf d}\in \mathbb{N}_+^{V(G_M)}$ be the vector defined in each $v\in V(G_M)$ as
\[
{\bf d}_v=\begin{cases}
1&\textrm{if }\displaystyle{\sum_{w\in V(G_M)} M_{v,w}}=0,\\
\ \\
\displaystyle{\sum_{w\in V(G_M)} M_{v,w}}&\textrm{otherwise}.
\end{cases}
\]
It is not difficult to see that $(G_M,{\bf d},{\bf 1})$ is an arithmetical graph, which is called the Laplacian arithmetical graph of $M$.

Now we prove the second statement of the theorem.

($\Rightarrow$) We proceed by contradiction.
Assuming that $M$ is reducible, without loss of generality we can suppose that
\[
M=\left(\begin{array}{cc}A&C\\0&B\end{array}\right)
\]
where $A$ and $B$ are square matrices of size $r\times r$ and $s \times s$ respectively and $A$ is irreducible.
First, if $C=0$, then
\[
\mathcal{A}(M)=\{(u\bff{r},v\bff{s}) \, | \, A\bff{r}^t=\bff{0}^t, (\bff{r}_1,\ldots, \bff{r}_r)=1, B\bff{s}^t=\bff{0}^t, ({\bf s}_1,\ldots, \bff{s}_s)=1, (u,v)=1\}
\]
is infinite.
Now, assume that $C\neq 0$ and let $\bff{d}\in\mathcal{A}_{\geq 1}(A)$ and $L=\mathrm{diag}(\bff{d})-A$.
If $(\mathrm{diag}(\bff{e})-B)\bff{s}=\bff{0}$ for some $(\bff{e},\bff{s})\in \mathbb{N}_+^s$ with $({\bf s}_1,\ldots, \bff{s}_s)=1$, then 
\[
(\mathrm{diag}(\bff{d},\bff{e})-M)(-vL^{-1}C\bff{s},v\bff{s})^t=\bff{0} \text{ for all } \bff{d}\in \mathcal{A}_{\geq 1}(A) \text{ and } v\in \mathbb{N}_+.
\]
Since $\bff{d}\in\mathcal{A}_{\geq 1}(A)$, $L$ is an irreducible non-singular $M$-matrix.
By~\cite[Theorem 6.2.7, page 141]{Berman}, $L^{-1}> 0$. 
Moreover, since $C\leq 0$, then $-L^{-1}C\bff{r}> 0$.
On the other hand, since $L$ is an integer, there exists $v\in \mathbb{N}_+$ such that 
$ -vL^{-1}C\bff{r}=-v\frac{1}{\mathrm{det}(\mathrm{det}(L))}\mathrm{Adj}(L)C\bff{r}$ is a vector of integers. 
Moreover, for all $\bff{d}\in\mathcal{A}_{\geq 1}(A)$, there exists $v\in \mathbb{N}_+$ such that 
the entries of $\frac{1}{u}(-vL^{-1}C\bff{s},v\bff{s})$ have greatest common divisor equal to one.
Finally, the result follows from the fact that $\mathcal{A}_{\geq 1}(A)$ is infinite.

\medskip

($\Leftarrow$) We claim that $\{\bff{d}\,| \, (\bff{d},\bff{r})\in \mathcal{A}(M)\}\subseteq \textrm{min}\left( \mathcal{A}_{\geq 0}(M)\right)$.
Let $(\bff{d},\bff{r})\in\mathcal{A}(M)$ and suppose that there exists $\mathbf{e}\in\mathcal{A}_{\geq0}(M)$ such that $\mathbf{e}<\mathbf{d}$.
If $\textrm{det}\big(\textrm{diag}(\bff{e})-M\big)>0$, we proceed as in the proof of Theorem~\ref{finitepositive}.
Otherwise, since $M$ is irreducible, $\textrm{diag}(\mathbf{e})-M$ is a singular and irreducible $M$-matrix.
Now, by~\cite[Theorem 6.4.16, page 156]{Berman}, $\textrm{diag}(\mathbf{e})-M$ is an almost non-singular $M$-matrix.
Thus by~Theorem \ref{almost}, $\textrm{det}\big(\textrm{diag}(\mathbf{d})-M\big)>\textrm{det}\big(\textrm{diag}(\mathbf{e})-M\big)=0$; 
which is a contradiction to the fact that $\mathbf{d}\in\mathcal{A}(M)$.
Thus, $\mathcal{A}(M)\subseteq \textrm{min}\left( \mathcal{A}_{\geq 0}(M)\right)$ and the result follows immediately from Dickson's lemma.
\end{proof}

Directly from Theorem~\ref{finitezero} we get the following corollary.
\begin{Corollary}\label{finitearithmetical}
If $G$ is a multidigraph, then $\mathcal{A}(G)$ is finite if and only if $G$ is strongly connected.
\end{Corollary}
\begin{proof}
Since $L(G,{\bf d})$ is an almost non-singular $M$-matrix for each ${\bf d}\in\mathcal{D}(G)$, 
it follows that $\mathcal{D}(G)\subseteq\mathcal{A}\big(G\big)$.
The result follows from Theorem \ref{finitezero} and the fact that $L(G,{\bf 0})$ is irreducible if and only if $G$ is strongly connected.
\end{proof}

\begin{Remark}\label{facts}
In~\cite[Proposition 1.1 and Corollary 1.3]{Lorenzini89} are given the following basic properties of Laplacian matrices for connected simple graphs.
Let $({\bf d},\bff{r})$ be an arithmetical structure of a connected multigraph $G=(V,E)$ and 
$M(u,v)$ be the $|V|-1$th minor of $M$ obtained by deleting the $u$th row and the $v$th column.
If $M=L(G,\bff{d})$, then
\begin{itemize}
\item[\it{(i)}] $M$ has rank $|V|-1$ and $\ker_{\mathbb{Q}}(M)=\langle \bff{r}\rangle$. 
\item[\it{(ii)}] There exists a positive integer $m$ such that $\mathrm{adj}(M)=m\bff{r}^t\bff{r}$. 
Furthermore, $m=\det(M(u,v))(\bff{r}_u \bff{r}_v)^{-1}$.
\item[\it{(iii)}] The cokernel of $M$ is a finitely generated abelian group of the form $\mathbb{Z}\oplus \Phi(G)$ 
where $\Phi(G)$ is a finite group of order $m$.
\end{itemize}
Note that the condition $\ker_{\mathbb{Q}}(M)=\langle \bff{r}\rangle$ implies that $\gcd(\bff{r}_v | v\in V)=1$. 
\end{Remark}

\begin{Example}
Let $G$ be the multidigraph illustrated in Figure \ref{fig1}. 
Then $G$ is not strongly connected since there is no directed path from the vertex $1$ to the vertex $4$.
Thus, $\mathcal{D}(G)$ has to be infinite.
In fact, $L(G,{\bf 0})$ is the matrix of Example \ref{Exampleinfinite}.

\begin{figure}[h]
\begin{tikzpicture}[line width=0.5pt, scale=0.8]
\draw[->] (1,-1) arc (315:228:1.425cm);
\draw[->] (1,-1) arc (315:402:1.425cm);
\draw[->] (-1,1) arc (135:48:1.425cm);
\draw[->] (-1,1) arc (135:222:1.425cm);

\draw[->] (1,1) arc (360:271:1.975cm);
\draw[->] (1,-1) arc (270:181:1.975cm);
\draw[->] (-1,-1) arc (180:91:1.975cm);
\draw[->] (-1,1) arc (90:1:1.975cm);

\draw {
	(1,1) node[draw, circle, fill=gray, inner sep=0pt, minimum width=3pt] (v1) {}
	(1,-1) node[draw, circle, fill=gray, inner sep=0pt, minimum width=3pt] (v2) {}
	(-1,-1) node[draw, circle, fill=gray, inner sep=0pt, minimum width=3pt] (v3) {}
	(-1,1) node[draw, circle, fill=gray, inner sep=0pt, minimum width=3pt] (v4) {} 
	
	(v1)+(0.2,0.2) node {$v_1$}
	(v2)+(0.22,-0.22) node {$v_2$}
	(v3)+(-0.2,-0.2) node {$v_3$}
	(v4)+(-0.22,0.22) node {$v_4$}		
};
\end{tikzpicture}
\caption{A not strongly connected multidigraph.}
\label{fig1}
\end{figure}
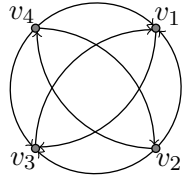
\end{Example}

\begin{Remark}
Clearly $\mathcal{D}(G)\subseteq\mathcal{A}_0\big(G\big)$, however we do not always have equality.
For instance, if $P_5$ with vertices $v_1,\ldots, v_5$, then 
\[
\left(
\begin{array}{ccccc}
1 & -1 & 0 & 0 & 0\\
-1 & 1 & -1 & 0 & 0\\
0 & -1 & d & -1 & 0\\
0 & 0 & -1 & 1 & -1\\
0 & 0 & 0 & -1 & 1
\end{array}
\right)
\left(
\begin{array}{c}
1\\
1\\
0\\
-1\\
-1
\end{array}
\right)
=\bff{0}.
\]
Therefore $\mathrm{det}(\mathrm{diag}(1,1,d,1,1)-A(P_5))=0$ for all $d\in \mathbb{N}_+$ and $\mathcal{D}(G)\subsetneq\mathcal{A}_0\big(G\big)$.
\end{Remark}

To finish this section we present some arithmetical structures on the cone of a graph.
This procedure represents the first example of how to construct arithmetical structures.
Given a graph $G$, let $c(G)$ be the cone of $G$, that is, the graph obtained from $G$ by adding a new vertex 
$v_c$ and all the edges between $v_c$ and the vertices of $G$.

\begin{Proposition}\label{cone1}
Let $G$ be a t-regular graph with $n$ vertices.
If $f$ is a divisor of $n$, then (\bff{d},\bff{r}) given by
\[
\bff{d}_u=(\frac{n}{f},t+f, \ldots, t+f) \text{ and } \bff{r}_u=(f,1, \ldots,1)
\]
is an arithmetical structure of $c(G)$.
\end{Proposition}
\begin{proof}
This follows because $\bff{d}\in \mathbb{N}_+^n$, $\bff{r}\in \mathbb{N}_+^n$ and
$\left(\begin{array}{cc}
\frac{n}{f} & -\bff{1}_n \\
-\bff{1}_n^t & L(G,(f+t)\bff{1}_n) \\
\end{array}\right)
\left(\begin{array}{c}
f\\
\bff{1}_n\\
\end{array}\right)=\bff{0}$.
\end{proof}

\begin{Example}
Let $G$ be the cycle with four vertices.
Then it is not difficult to see that with $n=4$ and $d=2$ we have
\[
\left(\begin{array}{ccccc}
2 & -1 & -1 & -1 & -1 \\
-1 & 4 & -1 & 0 & -1 \\
-1 & -1 & 4 & -1 & 0 \\
-1 & 0 & -1 & 4 & -1 \\
-1 & -1 & 0 & -1 &  4
\end{array}\right)
\left(\begin{array}{c}
2\\
1\\
1\\
1\\
1
\end{array}\right)=\bff{0}.
\]
\end{Example}

The next proposition gives us another type of arithmetical structure of the cone of a graph,
these arithmetical structures are more difficult to find.

\begin{Proposition}\label{cone2}
Let $G$ be a graph with $n$ vertices and $(\bff{d},\bff{r})\in \mathbb{N}_+^n\times \mathbb{N}_+^n$ 
such that $L(G,\bff{d})\bff{r}=a\bff{1}$ and $a\big| \sum_{i=1}^n \bff{r}_i=|\bff{r}|$. 
If $g=\mathrm{gcd}(a,\bff{r}_1,\ldots,\bff{r}_n)$, then $(\tilde{\bff{d}},\tilde{\bff{r}})$ given by
\[
\tilde{\bff{d}}_u=(\frac{\sum_{i=1}^n \bff{r}_i}{a},\bff{d}_1, \ldots, \bff{d}_n) \text{ and } \tilde{\bff{r}}_u=(\frac{a}{g}, \frac{\bff{r}_1}{g}, \ldots,\frac{\bff{r}_n}{g})
\]
is an arithmetical structure of $c(G)$.
\end{Proposition}
\begin{proof}
This follows because $\tilde{\bff{d}}\in \mathbb{N}_+^n$, $\tilde{\bff{r}}\in \mathbb{N}_+^n$ and
$\left(\begin{array}{cc}
\frac{|\bff{r}|}{a} & -\bff{1}_n \\
-\bff{1}_n^t & L(G,\bff{d}) \\
\end{array}\right)
\left(\begin{array}{c}
\frac{a}{g}\\
\\
\frac{\bff{\bff{r}^t}}{g}\\
\end{array}\right)=\bff{0}$.
\end{proof}


\section{arithmetical structures on the graph obtained by merging and splitting vertices}\label{mergesplit}
In this section we present a way to construct arithmetical structures for graphs obtained by merging and splitting vertices.
More precisely, we prove that if an arithmetical structure $(\bff{d},\bff{r})$ of a graph $G$ satisfies $\bff{r}_u=\bff{r}_v$ for some vertices $u$ and $v$,
then we can construct an arithmetical structure of the graph obtained by merging the vertices $u$ and $v$.
In a similar way, given an arithmetical structure $(\bff{d},\bff{r})$ of a graph $G$ and a vertex $u$ of $G$ that
satisfies $\bff{r}_u| \sum_{a\in A} \bff{r}_a$ for some $A\subsetneq N_{G}(u)$, we can construct an 
arithmetical structure of the graph obtained by splitting the vertex $u$.

Given a multidigraph $G$, an arithmetical structure $(\bff{d},\bff{r})$ of $G$, and $u,u'\in V(G)$ such that $\bff{r}_u=\bff{r}_{u'}$, 
let $m(G,u,u')$ be the graph obtained from $G$ by merging the vertices $u$ and $u'$ into a new vertex $w$.
Also, let
\[
m(\bff{d})_v=
\begin{cases}
\bff{d}_u+\bff{d}_{u'} & \text{ if } v=w,\\
\bff{d}_v & \text{ otherwise}
\end{cases}
\]
and $m(\bff{r})\in \mathbb{N}^{V(m(G,u,u'))}$ be given by $m(\bff{r})_v=\bff{r}_v$ for all $v\in V(m(G,u,u'))$.

\begin{Theorem}\label{merge}
If $(G,\bff{d},\bff{r})$ is an arithmetical graph such that $\bff{r}_u=\bff{r}_{u'}$ for some $u,u'\in V(G)$, then
$m(\bff{d},\bff{r})=(m(\bff{d}),m(\bff{r}))$ is an arithmetical structure of $m(G,u,u')$.
\end{Theorem}
\begin{proof}
Since $\bff{d}_u\bff{r}_u=\sum_{v\in N_G(u)}\bff{r}_v$ and $\bff{d}_{u'}\bff{r}_{u'}=\sum_{v\in N_G(u')}\bff{r}_v$,
\[
m(\bff{d})_w\bff{r}_w=\bff{d}_u\bff{r}_u+\bff{d}_{u'}\bff{r}_{u'}=\sum_{v\in N_G(u)}\bff{r}_v+\sum_{v\in N_G(u')}\bff{r}_v=\sum_{v\in N_{m(G,u,u')}(w)}m(\bff{r})_v
\] 
and therefore
{\small
\[
L(m(G,u,u'),m(\bff{d})) m(\bff{r})^t=
\left(\begin{array}{cccc}
\bff{d}_{u}+\bff{d}_{u'}& -1 & \cdots & 0 \\
-1& & & \\
\vdots& & L(G-\{u,u'\},\bff{d}|_{V(G)-\{u,u'\}}) & \\
0& & &
\end{array}\right)
\left(\begin{array}{c}
\bff{r}_u \\
\\
\bff{r}|_{V(G)-\{u,u'\}}
\end{array}\right)
=\bff{0}.\vspace{-7mm}
\]
}
\end{proof}

Before presenting the next result, we need to introduce some notation.
Given a multidigraph $G$, an arithmetical structure $(\bff{d},\bff{r})$ of $G$, and $u\in V(G)$ such that
\[
\bff{r}_u| \sum_{a\in A} \bff{r}_a \text{ for some } A\subsetneq N_{G}(u),
\]
let $s(G,u)$ be the graph obtained by splitting the vertex $v$ into two vertices $w$ and $w'$ with $N_{s(G,u)}(w)=A$ and $N_{s(G,u)}(w')=N_{G}(u)-A$.
Also, let
\begin{equation*}
\begin{array}{ccc}
s(\bff{d})_v=
\begin{cases}
\frac{\sum_{a\in A} \bff{r}_a}{\bff{r}_u} & \text{ if } v=w,\\
\frac{\sum_{a\in N_{G}(u)-A} \bff{r}_a}{\bff{r}_u} & \text{ if } v=w',\\
\bff{d}_v & \text{ otherwise},
\end{cases}
&\quad\textrm{and}\quad\quad&
s(\bff{r})_v=
\begin{cases}
\bff{r}_u & \text{ if } v=w,w',\\
\bff{r}_v & \text{ otherwise.}
\end{cases}
\end{array}
\end{equation*}

The next theorem is in some sense a converse of Theorem~\ref{merge}.

\begin{Theorem}\label{split}
If $(G,\bff{d},\bff{r})$ is an arithmetical graph such that
\[
\bff{r}_u| \sum_{a\in A} \bff{r}_a \text{ for some } A\subsetneq N_{G}(u) \text{ and } u\in V(G),
\]
then $s(\bff{d},\bff{r})=(s(\bff{d}),s(\bff{r}))$ is an arithmetical structure of $s(G,u)$.
\end{Theorem}
\begin{proof}
This follows directly from 
\[
L(s(G,u),s(\bff{d})) s(\bff{r})^t=
\left(\begin{array}{ccccc}
\bff{d}_{w}& -1 & \cdots &  &0 \\
-1& & & & \vdots \\
\vdots& & L(G-\{u\},\bff{d}|_{V(G)-\{u\}}) & & -1\\
0& & \cdots  & -1 & \bff{d}_{w'}
\end{array}\right)
\left(\begin{array}{c}
\bff{r}_u \\
\bff{r}|_{V(G)-\{u\}} \\
\bff{r}_u
\end{array}\right)
=\bff{0}. \vspace{-6mm}
\]
\end{proof}


\section{arithmetical structures on the clique--star transform of a graph}\label{sclique--star}
In this section we study the arithmetical structures on the clique--star transform of a graph.
Given a graph $G$ and a clique $C$ (a set of pairwise adjacent vertices) of $G$, 
the clique--star transform of G, denoted by $cs(G,C)$, 
is the graph obtained from $G$ by deleting all the edges between the vertices in $C$
and adding a new vertex $v$ with all the edges between $v$ and the vertices in $C$, see Figure~\ref{figcs}.
The clique--star transformation generalizes the subdivision of an edge, adding of pendant edges, and the $\Delta-Y$ transformation on graphs.
We establish a relationship between the arithmetical structures on $G$ and $cs(G,C)$
and prove that the critical groups of $G$ and $cs(G,C)$ are isomorphic.
Using this relation we can describe completely the arithmetical structures on the paths and cycles, 
see Sections~\ref{ArithPath} and~\ref{ArithCycle}.

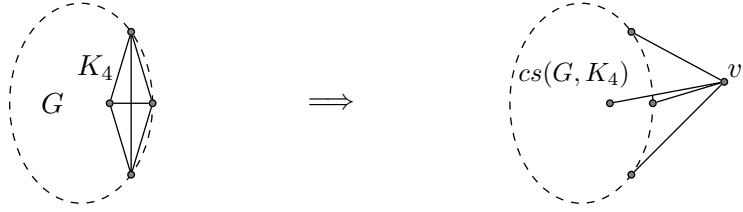
\begin{figure}[h]
\begin{tikzpicture}[line width=0.5pt, scale=0.95]
\tikzstyle{every node}=[inner sep=0pt, minimum width=2.5pt] 
\filldraw[fill=white,dashed] (0,5) ellipse (10mm and 14mm);
\draw (1,5) node (v1) [draw, circle, fill=gray] {};
\draw (0.4,5) node (v3) [draw, circle, fill=gray] {};
\draw (0.7,6) node (v2) [draw, circle, fill=gray] {};
\draw (0.7,4) node (v4) [draw, circle, fill=gray] {};
\draw (v4) -- (v1) -- (v2) -- (v3) -- (v4) -- (v2);
\draw (v1) -- (v3);
\draw (-0.4,5) node {$G$};
\draw (0.2,5.5) node {$K_4$};
\draw (3.5,5) node {$\Longrightarrow$};

\filldraw[fill=white,dashed] (0+7,5) ellipse (10mm and 14mm);
\draw (1+7,5) node (v1) [draw, circle, fill=gray] {};
\draw (0.4+7,5) node (v2) [draw, circle, fill=gray] {};
\draw (0.7+7,6) node (v3) [draw, circle, fill=gray] {};
\draw (0.7+7,4) node (v4) [draw, circle, fill=gray] {};
\draw (2+7,5.3) node (v) [draw, circle, fill=gray] {};
\draw (v) -- (v1);
\draw (v) -- (v2);
\draw (v) -- (v3);
\draw (v) -- (v4);
\draw (-0.1+7,5.4) node {\small $cs(G,K_4)$};
\draw (v)+(0.15,0.15) node {$v$};
\end{tikzpicture}
\caption{A graph $G$ with a clique with four vertices and their clique--star transform $cs(G, K_4)$}\label{figcs}
\end{figure}

The next theorem gives us the relation between the arithmetical structures on $G$ and $cs(G,C)$.
Before establishing the theorem we will fix some notation.
Given $(\bff{d},\bff{r})\in \mathcal{A}(G)$, let
\begin{equation*}
\begin{array}{ccc}
\tilde{\bff{d}}_u=cs(\bff{d},C)_u=
\begin{cases}
{\bf d}_u&\textrm{ if } u\not\in C,\\
{\bf d}_u+1&\textrm{ if } u\in C,\\
1&\textrm{ if }u=v,
\end{cases}
&\quad\textrm{and}\quad\quad&
\tilde{\bff{r}}_u=cs(\bff{r},C)_u=
\begin{cases}
\bff{r}_u&\textrm{ if } u\in V\\
\\
\displaystyle{\sum_{u\in C} \bff{r}_u}&\textrm{ if }u=v,
\end{cases}
\end{array}
\end{equation*}
where $V$ is the set of vertices of $G$.

\begin{Theorem}\label{clique--star}
Let $G$ be a graph, $C$ a clique of $G$, and $\tilde{G}=cs(G,C)$.
If $\mathcal{A}'(\tilde{G})$ is the set of arithmetical structures $(\bff{d}',\bff{r}')$ of $\tilde{G}$ with $\bff{d}'_v=1$, then
\[
\mathcal{A}'(\tilde{G})=\{(\tilde{\bff{d}},\tilde{\bff{r}}) \,|\, (\bff{d},\bff{r})\in \mathcal{A}(G)\}.
\] 
\end{Theorem}
\begin{proof}
Let $V$ be the set of vertices of $G$, and let $T_C$ be the trivial graph with vertex set equal to $C$. Put $|V|=n$ and $|C|=c$.
By the definitions of $\tilde{\bff{d}}$ and $\tilde{\bff{r}}$,
\[
L(\tilde{G},\tilde{\bff{d}})\tilde{\bff{r}}^t= 
\left(\begin{array}{ccc}
1&-\bff{1}_{c}& \bff{0}_{n-c}\\
-\bff{1}_{c}^t & L(T_C,\tilde{\bf{d}}|_C) & * \\
\bff{0}_{n-c}^t&  * & L(G[V-C],\tilde{\bf{d}}|_{V-C})
\end{array}\right)
\tilde{\bff{r}}^t
=0
\] 
if and only if $\tilde{\bff{r}}_v=\sum_{u\in C} \tilde{\bff{r}}_u$,
\[
-\tilde{\bff{r}}_v+\tilde{\bff{r}}_u\tilde{\bff{d}}_u+\sum_{w\in V-C} L(G,\tilde{\bf{d}})_{u,w}\tilde{\bff{r}}_w=0\text{ for all }u\in C 
\text{ and }\sum_{w\in V} L(G,\tilde{\bf{d}})_{u,w}\tilde{\bff{r}}_w=0\text{ for all }u\in V-C
\] 
if and only if 
\[
\bff{r}_u\bff{d}_u-\sum_{w\in C-u} \bff{r}_w+\sum_{w\in V-C} L(G, \bff{d})_{u,w}\bff{r}_w=0\text{ for all }u\in C 
\text{ and } \sum_{w\in V} L(G, \bff{d})_{u,w}\bff{r}_w=0\text{ for all }u\in V-C
\]
if and only if $L(G, \bff{d})\bff{r}^t=0$.
\end{proof}

Lorenzini in~\cite[page 485]{Lorenzini89} introduced a similar operation, called the blowup, of an arithmetical structure of a graph. 
The blowup generalizes the clique--star transformation of a graph.
However, the blowup does not always have a meaning in the context of graphs.
Here we present a slightly more general variant of Lorenzini's construction.

Given a non-negative integral $n\times n$ matrix $M$ with all its diagonal entries equal to zero,
$\bff{p}=\{p_1\ldots, p_n\}\in \mathbb{N}^n$ and $\bff{q}=\{q_1\ldots, q_n\}\in \mathbb{N}^n$
with $g=\mathrm{gcd}(p_1\ldots, p_n, q_1\ldots, q_n)$, let
\[
\tilde{M}=b_{\bff{p},\bff{q}}(M)=
\left(\begin{array}{cc}
1 & -\bff{q} \\
-\bff{p} & \frac{[\bff{p}^t\bff{q}]}{g}+\mathrm{diag}(\bff{d})-M
\end{array}\right),
\]
where $[\bff{p}^t\bff{q}]=\bff{p}^t\bff{q}-\mathrm{diag}(p_1q_1,\ldots,p_nq_n)$.
For simplicity we enumerate the rows and columns of $b_{\bff{p},\bff{q}}(M)$ from $0$ to $n$ instead of from $1$ to $n+1$.
Also, given an arithmetical structure $(\bff{d},\bff{r})$ of $M$, let
\begin{equation*}
\begin{array}{ccc}
\tilde{\bff{d}}_i=b_{\bff{p},\bff{q}}(\bff{d})_i=\begin{cases}
g&\textrm{ if } i=0,\\
{\bf d}_i+p_iq_i&\textrm{ otherwise,}
\end{cases}
&\quad\textrm{and}\quad\quad&
\tilde{\bff{r}}_i=b_{\bff{p},\bff{q}}(\bff{r})_i=\begin{cases}
\displaystyle{\sum_{j=1}^n \frac{\bff{r}_j\bff{q}_j}{g}}&\textrm{ if }i=0,\\
\bff{r}_i&\textrm{ otherwise.}
\end{cases}
\end{array}
\end{equation*}

\begin{Theorem}\label{blowup}
Let $M$ be a non-negative integral $n\times n$ matrix with all diagonal entries equal to zero, 
$\bff{p}=\{p_1\ldots, p_n\}\in \mathbb{N}^n$ and $\bff{q}=\{q_1\ldots, q_n\}\in \mathbb{N}^n$
with $g=\mathrm{gcd}(p_1\ldots, p_n, q_1\ldots, q_n)$.
If $g\neq 0$, $\bff{p},\bff{q}\neq \bff{0}$ and $\mathcal{A}'(\tilde{M})$ is the set of arithmetical structures $(\bff{d}',\bff{r}')$ of $\tilde{M}$ with $\bff{d}'_v=g$, then
\[
\mathcal{A}'(\tilde{M})=\{(\tilde{\bff{d}},\tilde{\bff{r}}) \,|\, (\bff{d},\bff{r})\in \mathcal{A}(M)\}.
\]
\end{Theorem}
\begin{proof}
This follows by using similar arguments to those given in the proof Theorem~\ref{clique--star}.
Note that
\[
\left(\begin{array}{cc}
1 & -\bff{q} \\
-\bff{p} & \frac{\bff{p}^t\bff{q}}{g}+\mathrm{diag}(\bff{d})-M
\end{array}\right).
\]
Since $g\tilde{\bff{r}}_0=\sum_{i=1}^n\bff{r}_i\bff{q}_i=\bff{q}\bff{r}^t$ and 
$\bff{p}\tilde{\bff{r}}_0=(\sum_{j=1}^n\frac{\bff{p}_1\bff{q}_j\bff{r}_j}{g}, \ldots, \sum_{j=1}^n\frac{\bff{p}_n\bff{q}_j\bff{r}_j}{g})=\frac{\bff{p}^t\bff{q}}{g}\bff{r}^t$ 
for all $1\leq i\leq n$, we have $(\mathrm{diag}(\tilde{\bff{d}})-\tilde{M})\tilde{\bff{r}}^t=0$ if and only if $(\mathrm{diag}(\bff{d})-M)\bff{r}^t=0$.
\end{proof}

Note that we recover Theorem~\ref{clique--star} from Theorem~\ref{blowup} by putting $\bff{p}$ and $\bff{q}$ equal to the characteristic vector of the clique $C$.

\begin{Theorem}\label{criticoblowup}
Let $M$ be a non-negative integral $n\times n$ matrix with all diagonal entries equal to zero, 
$\bff{p}=\{p_1\ldots, p_n\}\in \mathbb{N}^n$ and $\bff{q}=\{q_1\ldots, q_n\}\in \mathbb{N}^n$
with $g=\mathrm{gcd}(p_1\ldots, p_n, q_1\ldots, q_n)$.
If $g=1$, $\bff{p},\bff{q}\neq \bff{0}$ and $(\bff{d},\bff{r})\in \mathcal{A}(M)$, then
\[
K(\tilde{M},\tilde{\bff{d}},\tilde{\bff{r}})\cong K(M,\bff{d},\bff{r}).
\]
\end{Theorem}
\begin{proof}
This follows since $\mathrm{diag}(\tilde{\bff{d}})-\tilde{M}$ is integrally equivalent to $\left(\begin{array}{cc}
1 & \bff{0}\\
\bff{0} & \mathrm{diag}(\bff{d})-M
\end{array}\right)$.
More precisely,
\[
\left(\begin{array}{cc}
1 & \bff{0}\\
\bff{p} & I_{n}
\end{array}\right)
\left(\begin{array}{cc}
1 & -\bff{q} \\
-\bff{p} & \frac{\bff{p}^t\bff{q}}{g}+\mathrm{diag}(\bff{d})-M
\end{array}\right)
\left(\begin{array}{cc}
1 & \bff{q}\\
\bff{0} & I_{n}
\end{array}\right)
=
\left(\begin{array}{cc}
1 & \bff{0}\\
\bff{0} & \mathrm{diag}(\bff{d})-M
\end{array}\right).\vspace{-5mm}
\]
\end{proof}

The applications of Theorem~\ref{clique--star} are very wide.
For instance, if we apply Theorem~\ref{clique--star} to the complete graph $K_n$ with $C=V(K_n)$ for $n\geq 2$, 
we obtain that $cs(K_n,C)$ is the star $S_n$ with $n$ leaves.
If $v$ is the center of $S_n$ and we label the leaves of $S_n$ from $1$ to $n$, then by~\cite{twin} 
\[
\mathcal{A}'(S_n)=\{(\bff{d},\bff{r})\in\mathbb{N}^{n+1}\times \mathbb{N}^{n+1} \, | \, \bff{d}_v=1=\sum_{i=1}^n \frac{1}{\bff{d}_i}, 
\bff{r}_v=c\text{ and } \bff{r}_i=\frac{c}{\bff{d}_i} \},
\]
where $c=\mathrm{lcm}(\bff{d}_1,\ldots,\bff{d}_n)$.
For a complete description of the arithmetical structures on a star, see~\cite[Corollary 3.1]{1connected}.

In a similar way, let $K_{2,n}$ be the complete bipartite graph with bipartition $V_1=\{v_1,v_2\}$ and $V_2=\{u_1,\ldots,u_n\}$ of size $2$ and $n$ respectively. 
Applying Theorem~\ref{clique--star} to the complete graph $K_{n+1}$ with vertex set $\{v_2,u_1,\ldots,u_n\}$ and $C=\{u_1,\ldots,u_n\}$, we get
\begin{eqnarray*}
\mathcal{A}'(K_{2,n})=
\big\{(\bff{d},\bff{r})\in\mathbb{N}^{n+2}\times \mathbb{N}^{n+2} \, | \!\!\! &&\!\!\!\!\!\! \bff{d}_{v_1}=1, (1+\frac{1}{\bff{d}_{v_2}})
\left(\sum_{i=1}^n \frac{1}{\bff{d}_{u_i}}\right)=1,\\
\!\!\!&&\!\!\!\!\! \bff{r}_{v_1}=c\left(\sum_{i=1}^n \frac{1}{\bff{d}_{u_i}}\right), \,\bff{r}_{v_2}=\frac{c}{\bff{d}_{v_2}+1}\text{ and } \bff{r}_{u_i}=\frac{c}{\bff{d}_i} \big\},
\end{eqnarray*}
where $c=\mathrm{lcm}(\bff{d}_1,\ldots,\bff{d}_n)$.
Note that $(1+\frac{1}{\bff{d}_{v_2}})\left(\sum_{i=1}^n \frac{1}{\bff{d}_{u_i}}\right)=1$ if and only if $\frac{1}{\bff{d}_{v_2}+1}+\sum_{i=1}^n \frac{1}{\bff{d}_{u_i}}=1$.
For a complete description of the arithmetical structures on a complete bipartite graph see~\cite[Corollary 3.1]{twin}.
\begin{Example}
Taking $n=3$ we have that $\bff{d}=(1,2,3,4,12)$ and $\bff{r}=(8,4,4,3,1)$.
\[
L(K_{2,n},\bff{d})\bff{r}=
\left(\begin{array}{ccccc}
1 & 0 & -1 &-1 &-1\\
0 & 2 & -1 &-1 &-1\\
-1 & -1 & 3 & 0 & 0\\
-1 & -1 & 0 & 4 & 0\\
-1 & -1 & 0 & 0 & 12
\end{array}\right)
\left(\begin{array}{c}
8\\
4\\
4\\
3\\
1
\end{array}\right)=\bff{0}.
\]
\end{Example}

In a general sense we have two special cases of the clique--star transformation, the subdivision of an edge and adding a pendant edge.
These cases will play a very important role in Sections~\ref{ArithPath} and~\ref{ArithCycle}.
In the following we will present explicit forms of Theorem~\ref{clique--star} in the cases of adding a pendant edge and the subdivision of an edge.
\begin{Corollary}\label{addition}
Given a graph $G$ and $v$ one of its vertices, let $\tilde{G}$ be the graph resulting from adding the edge $vv'$. 
If $({\bf d},\bff{r})$ is an arithmetical structure of $G$, then
\begin{equation*}
\begin{array}{ccc}
\tilde{\bff{d}}_u=\begin{cases}
1&\textrm{ if }u=v',\\
{\bf d}_u+1&\textrm{ if } u=v,\\
{\bf d}_u&\textrm{ otherwise}.
\end{cases}
&\quad\textrm{and}\quad\quad&
\tilde{\bff{r}}_u=\begin{cases}
\bff{r}_v&\textrm{ if }u=v'\\
\\
\bff{r}_u&\textrm{ otherwise}
\end{cases}
\end{array}
\end{equation*}
is an arithmetical structure of $\tilde{G}$.
\end{Corollary}
\begin{proof}
This follows from Theorem~\ref{clique--star} with $C=\{v\}$.
\end{proof}

\begin{Corollary}\label{subdivision}
Given a graph $G$ and $e=u_1u_2$ one of its edges, let $\tilde{G}$ be the graph obtained by subdividing the edge $e$, see Figure~\ref{figsubdivision}. 
If $({\bf d},\bff{r})$ is an arithmetical structure of $G$, then
\begin{equation*}
\begin{array}{ccc}
\tilde{\bff{d}}_u=\begin{cases}
1&\textrm{ if }u=v,\\
{\bf d}_u+1&\textrm{ if } u=u_1,u_2,\\
{\bf d}_u&\textrm{ otherwise},
\end{cases}
&\quad\textrm{and}\quad\quad&
\tilde{\bff{r}}_u=\begin{cases}
\displaystyle{\sum_{u\in e} \bff{r}_u}&\textrm{ if }u=v,\\
\\
\bff{r}_u&\textrm{ otherwise},
\end{cases}
\end{array}
\end{equation*}
is an arithmetical structure of $\tilde{G}$.
\end{Corollary}
\begin{proof}
This follows from Theorem~\ref{clique--star} with $C=\{u_1,u_2\}$.
\end{proof}
The arithmetical structure $(\tilde{\bff{d}}, \tilde{\bff{r}})$ of $\tilde{G}$ is called the edge subdivision of $({\bf d},\bff{r})$.
Figure~\ref{figsubdivision} illustrates the graph obtained by the subdivision of an edge.
\begin{figure}[h]
\begin{tikzpicture}[scale=0.9]
\draw (0,5-4)--(2,5-4+1);
\draw (0,5-4+1.25)--(2,5-4+1);
\draw (0,5-4-1.25)--(2,5-4+1);
\draw (0,5-4)--(2,5-4-1);
\draw (0,5-4+1.25)--(2,5-4-1);
\draw (0,5-4-1.25)--(2,5-4-1);
\draw (2,5-4+1)--(2,5-4-1);
\filldraw[fill=white,dashed] (0,5-4) ellipse (25pt and 40pt);
\filldraw[fill=white] (2,5-4+1) circle (12pt);
\filldraw[fill=white] (2,5-4-1) circle (12pt);
\draw (2,5-4+1) node {$d_1$};
\draw (2,5-4-1) node {$d_2$};
\draw (5,1) node {$\Longrightarrow$};
\draw (0,1) node {$G$};
\draw (2.2,1) node {$e$};

\draw (0+8,5-4)--(2+8,5-4+1);
\draw (0+8,5-4+1.25)--(2+8,5-4+1);
\draw (0+8,5-4-1.25)--(2+8,5-4+1);
\draw (0+8,5-4)--(2+8,5-4-1);
\draw (0+8,5-4+1.25)--(2+8,5-4-1);
\draw (0+8,5-4-1.25)--(2+8,5-4-1);
\draw (2+8,5-4+1)--(4+8,5-4);
\draw (2+8,5-4-1)--(4+8,5-4);
\filldraw[fill=white,dashed] (0,5-4) (0+8,5-4) ellipse (25pt and 40pt);
\filldraw[fill=white] (2+8,5-4+1) circle (12pt);
\filldraw[fill=white] (2+8,5-4-1) circle (12pt);
\filldraw[fill=white] (4+8,5-4) circle (12pt);
\draw (2+8,5-4+1) node {\small $d_1\!\!+\!\!1$};
\draw (2+8,5-4-1) node {\small $d_2\!\!+\!\!1$};
\draw (4+8,5-4) node {$1$};
\draw (10,1) node {$\tilde{G}$};
\end{tikzpicture}
\caption{A graph $G$ with an edge $e$ and the graph $\tilde{G}$ obtained by the subdivision of $e$.}\label{figsubdivision}
\end{figure}
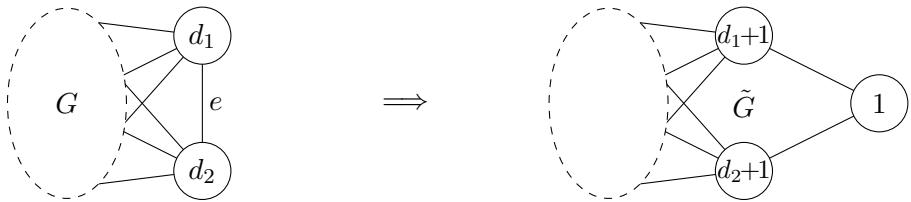

\newpage
It is not difficult to see that $P_n$ can be obtained from $P_{n-1}$ by the subdivision of an edge or by adding a pendant edge.
Thus, it immediately results that at least a part of the arithmetical structures on $P_n$ can be obtained from the arithmetical structures on $P_{n-1}$.
Moreover, it can be proved that all the the arithmetical structures on $P_n$ can be obtained in this way, see Section~\ref{ArithPath}.
The non-trivial part is to prove that the whole of $\mathcal{D}(P_n)$ can be obtained from $\mathcal{D}(P_{n-1})$.
Before doing this we illustrate Corollaries~\ref{addition} and~\ref{subdivision} with some examples.

\begin{Example}\label{FromP2toP3}
It is easy to prove that $\mathcal{D}(P_2)=\{(1,1)^t\}$. 
Thus, applying Corollary~\ref{addition} to $P_2$ we get that $(1,2,1)^t\in \mathcal{D}(P_3)$ (because ${P_2}'_v=P_3$ for any $v\in V(P_2)$).
Also, if we apply Corollary~\ref{subdivision} to the unique edge of $P_2$, we obtain $(2,1,2)^t\in \mathcal{D}(P_3)$  (because ${P_2}'_e=P_3$).
Since $\textrm{det}(L(P_3,{\bf d}))=d_1d_2d_3-d_1-d_3$, it is not difficult to prove that indeed $\mathcal{D}(P_3)=\{(1,2,1)^t,(2,1,2)^t\}$.

By the same procedure it can be found that
\[
\{ (1,2,2,1)^t,(1,3,1,2)^t,(2,1,3,1)^t,(3,1,2,2)^t,(2,2,1,3)^t\}\subseteq \mathcal{D}(P_4)
\]
But in this case it is much more difficult to prove that these are all the arithmetical structures on $P_4$.
\end{Example}

\begin{Example}\label{FromC3toC4}
Since $C_4$ is the subdivision of $C_3$ by any edge, we can apply Corollary~\ref{subdivision} to get arithmetical structures on $C_4$.
Fix the natural labeling on $C_4$, since $\left((1,2,5),(3,2,1)\right)$ is an arithmetical structure of $C_3$ we get that
\[
\left((1,2,2,6),(4,3,2,1)\right), \left((2,1,3,5),(3,5,2,1)\right), \left((1,3,1,6),(3,2,3,1)\right)
\]
are arithmetical structures on $C_4$.

In fact, for any $n\geq 4$, $C_n$ is a subdivision of $C_{n-1}$.
This fact allows us to obtain arithmetical structures on $C_n$ by subsequent applications of Corollary~\ref{subdivision} to the arithmetical structures on $C_3=K_3$.
Moreover, we will prove in Section~\ref{ArithCycle} that if $n\geq 4$, any arithmetical structure of $C_n$ different 
from $(\bff{d},\bff{r})=\left(2\cdot{ \bf 1},{ \bf 1}\right)$ can be obtained in this way, see Theorem~\ref{DCnrecursive}.
\end{Example}


\section{Applications: arithmetical structures on paths and cycles}\label{app}
In this section we give a recursive description of the arithmetical structures on a path and cycle, 
see Theorems~\ref{DPnrecursive} and~\ref{DCnrecursive}.
From these descriptions we get algorithmic descriptions of the arithmetical structures on 
paths and cycles, see Algorithm~\ref{AlgoGen} and Corollaries~\ref{DPnalgo} and~\ref{DCnalgo}.
The recursive descriptions have the advantage of being very short, 
but the disadvantage that producing the arithmetical structures on a path when it has many vertices is hard. 
An advantage of the algorithmic approach is that is easy to implement 
and allows constructing arithmetical structures with prescribed entries equal to $1$ on the $\bff{d}$.
At the end of the section we study the arithmetical structures on the subdivision 
of the complete graph that can be obtained using Algorithm~\ref{AlgoGen}.


\subsection{arithmetical structures on a path}\label{ArithPath}

We begin by given a recursive way to find all the arithmetical structures on the path with $n$ vertices.

\begin{Theorem}\label{DPnrecursive}
If $(\bff{d},\bff{r})$ is an arithmetical structure of $P_n$ for $n\geq 3$ and ${\bf d}\neq (1,2,\ldots,2,1)$,
then there exists a non-terminal vertex $v$ of $P_n$ such that 
\[
{\bf d}_v=1 \text{ and }{\bf d}_u>1\text{ for all }u\in N_{P_n}(v).
\]
Moreover, any arithmetical structure $(\bff{d},\bff{r})$ of $P_n$ different from the Laplacian
can be obtained from an arithmetical structure of $P_{n-1}$ by subdividing an edge.
\end{Theorem}
\begin{proof}
Let $P_n=v_1v_2\cdots v_n$.
Since $(\bff{d}_c,\bff{r}_c)=((1,2,\ldots,2,1),\bff{1})\in \mathcal{A}(P_{n})$ and $\bff{d}>0$, 
Theorem~\ref{almost} implies that there exists $1< i<n$ such that $\bff{d}_{v_i}=1$.
Since $L(P_n,\bff{d})\bff{r}^t=\bff{0}^t$, $\bff{r}_{v_i}=\bff{r}_{v_{i-1}}+\bff{r}_{v_{i+1}}$.
Assume that $3\leq i\leq n-2$. 
If ${\bf d}_{v_{i-1}}=1$, then $\bff{r}_{v_{i-1}}=\bff{r}_{v_{i-2}}+\bff{r}_{v_{i}}$ 
and therefore $\bff{r}_{v_{i-2}}+\bff{r}_{i+1}=0$; which is a contradiction to the fact that $\bff{r}>0$.
In a similar way it can be proved that ${\bf d}_{v_{i+1}}>1$.
Now, if $i=2$, then $\bff{r}_1=\bff{r}_{2}$ and therefore $\bff{r}_{3}=0$; a contradiction.
The case $i=n-1$ is similar.

Now, let $(\bff{d}',\bff{r}')\in \mathbb{N}_+^{V(P_n)-v_i}\times \mathbb{N}_+^{V(P_n)-v_i}$ be given by $\bff{r}'=\bff{r}|_{V(P_n)-v_i}$
\[
\bff{d}'_u=\begin{cases}
\bff{d}_u-1&\textrm{ if }u=v_{i-1},v_{i+1},\\
{\bf d}_u&\textrm{ otherwise}.
\end{cases}
\]
Since $L(P_n,{\bf d})\bff{r}^t={\bf 0}^t$, then $\bff{r}_{v_i}=\bff{r}_{v_{i-1}}+\bff{r}_{v_{i+1}}$,  
$\bff{r}_{v_{i-1}}\bff{d}_{v_{i-1}}=\bff{r}_{v_{i-2}}+\bff{r}_{v_{i}}$, $\bff{r}_{v_{i+1}}\bff{d}_{v_{i+1}}=\bff{r}_{v_{i}}+\bff{r}_{v_{i+2}}$.
Then 
\begin{eqnarray*}
\bff{r}'_{v_{i-1}}\bff{d}'_{v_{i-1}}&=&\bff{r}_{v_{i-1}}(\bff{d}_{v_{i-1}}-1)=\bff{r}_{v_{i-2}}+\bff{r}_{v_{i+1}}=\bff{r}'_{v_{i-2}}+\bff{r}'_{v_{i+1}},\\
\bff{r}'_{v_{i+1}}\bff{d}'_{v_{i+1}}&=&\bff{r}_{v_{i+1}}(\bff{d}_{v_{i+1}}-1)=\bff{r}_{v_{i-1}}+\bff{r}_{v_{i+2}}=\bff{r}'_{v_{i-1}}+\bff{r}'_{v_{i+2}}.
\end{eqnarray*}
Furthermore, since $\bff{d}',\bff{r}'>\bff{0}$, $(\bff{d}',\bff{r}')$ is an arithmetical structure of $P_{n-1}$.
Finally, the result follows by applying Corollary~\ref{subdivision}. 
The cases $i=2,n-1$ are similar.
\end{proof}

Note that the Laplacian arithmetical structure of $P_{n}$ can be obtained from the Laplacian arithmetical structure of $P_{2}$ by adding pendant vertices.

As we mentioned before, in practice, Theorem~\ref{DPnrecursive} does not give a good way for constructing arithmetical structures on $P_n$ when $n$ is large.
With this in mind, in the following we present an algorithm that produces the arithmetical $\bff{r}$-structures (that is the $\bff{r}$ vector of an arithmetical structure $(\bff{d},\bff{r})$) on $P_n$.
Moreover, we present this algorithm in a very general way, which allow us to apply it in order to construct arithmetical structures on any subdivision of a graph.
It is important to note that in the special case of $P_n$ and $C_n$, it is possible to give a direct formula for the output of this algorithm.
However, we prefer the algorithmic approach because, at least from our point of view, it seems more natural and more general.

Before presenting the algorithm, let us introduce some notation.
In the following, $V=V(G)$ and $e_v$ will denote the vector in $\mathbb{N}_+^{V}$ given by $(e_v)_u=\delta_{v,u}$ for each $v\in V$.
Given $U\subseteq V$ with $m$ vertices, an \emph{order} on $U$ is a bijective function $\theta: U \rightarrow [m]$.
Note that $\theta$ can be described by the vector $(\theta^{-1}(1),\ldots,\theta^{-1}(m))$.
Given $u,v\in V$, let $P_{u,v}$ be the set of disjoint paths between $u$ and $v$.
Also, given $\bff{r}\in \mathbb{N}^{V}$ and $u\in V$, let
\[
W_u(\bff{r})=\{ w\in V\, | \, \exists\, P\in P_{u,w}\text{ such that } \bff{r}_w\neq 0 \text{ and } \bff{r}_v=0 \text{ for all } v \text{ internal vertex of }P \}.
\]
Note that $W_u(\bff{r})$ can be a multiset.
In the special case when $G$ is a path or a cycle, the description of $W_u(\bff{r})$ can be simplified.
Finally, given a vector $\bff{r}\in \mathbb{N}^{V}$, let 
\[
\mathrm{supp}(\bff{r})=\{v\in V\, | \, \bff{r}_v\neq 0\}.
\]
The algorithm will receive three inputs: a finite graph $G$, a vector $\bff{r}_0\in \mathbb{N}^{V(G)}$ with $\mathrm{supp}(\bff{r}_0)=V(G)-U$,
and an order $\theta$ on $U$.

\begin{table}[h]\label{AlgoGen}
\caption{\ }
\begin{tabularx}{400pt}{>{\hsize=\hsize}X}
\begin{itemize}
\item[\upshape Input:] \begin{itemize} \item[] A graph $G$, a vector $\bff{r}_0\in \mathbb{N}^{V(G)}$ with $\mathrm{supp}(\bff{r}_0)=V-U$, 
and an order $\theta$ on $U$.\end{itemize}
\vspace*{10pt}
\item[\upshape Output:] \begin{itemize} \item[] A vector $\bff{r}=\bff{r}(G,\bff{r}_0,\theta)\in \mathbb{N}_+^V$. \end{itemize}
\vspace*{10pt}
\item[\upshape Algorithm:]
\begin{itemize} 
	\item[] Set $i=1$ and $\bff{r}=\bff{r}_0$.
	\item[]
	\item[] While $i\leq |U|$, set $u=\theta^{-1}(i)$, $i=i+1$, and $\bff{r}=\displaystyle{\bff{r}+\sum_{w\in W_u(\bff{r})} \bff{r}_w }\ e_{u}$.
	\item[] Return $\bff{r}=\bff{r}(G,\bff{r}_0,\theta)$.
\end{itemize}
\end{itemize}
\end{tabularx}
\end{table}

Note that $W_u(\bff{r})$ is not always well defined. However, it is well defined for the cases in which we are interested in this section.
Now we present an example to illustrate the algorithm. 
For simplicity, we often label the vertices of $G$ with the set $\{1,\ldots,n\}$.

\begin{Example}
Let $P_5$ be the path with vertices labeled by $1,2,3,4,5$, $\bff{r}_0=(1,0,0,0,1)$ and $\theta=(3,2,1)$.
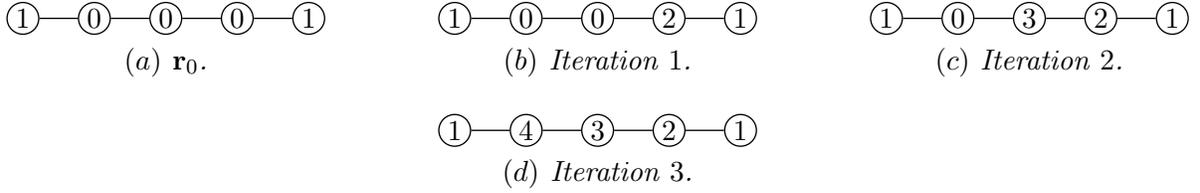
\begin{figure}[h]
\begin{center}
\begin{tabular}{c@{\hskip 15mm}c@{\hskip 15mm}c}
\begin{tikzpicture}[line width=0.5pt, scale=0.95]
\tikzstyle{every node}=[inner sep=0pt, minimum width=12pt] 
\draw (0,0) node (v1) [draw, circle] {};
\draw (1,0) node (v2) [draw, circle] {};
\draw (2,0) node (v3) [draw, circle] {};
\draw (3,0) node (v4) [draw, circle] {};
\draw (4,0) node (v5) [draw, circle] {};
\draw (v1) -- (v2) -- (v3) -- (v4) -- (v5);
\draw (v1) node {$1$};
\draw (v2) node {$0$};
\draw (v3) node {$0$};
\draw (v4) node {$0$};
\draw (v5) node {$1$};
\end{tikzpicture}
&
\begin{tikzpicture}[line width=0.5pt, scale=0.95]
\tikzstyle{every node}=[inner sep=0pt, minimum width=12pt] 
\draw (0,0) node (v1) [draw, circle] {};
\draw (1,0) node (v2) [draw, circle] {};
\draw (2,0) node (v3) [draw, circle] {};
\draw (3,0) node (v4) [draw, circle] {};
\draw (4,0) node (v5) [draw, circle] {};
\draw (v1) -- (v2) -- (v3) -- (v4) -- (v5);
\draw (v1) node {$1$};
\draw (v2) node {$0$};
\draw (v3) node {$0$};
\draw (v4) node {$2$};
\draw (v5) node {$1$};
\end{tikzpicture}
&
\begin{tikzpicture}[line width=0.5pt, scale=0.95]
\tikzstyle{every node}=[inner sep=0pt, minimum width=12pt] 
\draw (0,0) node (v1) [draw, circle] {};
\draw (1,0) node (v2) [draw, circle] {};
\draw (2,0) node (v3) [draw, circle] {};
\draw (3,0) node (v4) [draw, circle] {};
\draw (4,0) node (v5) [draw, circle] {};
\draw (v1) -- (v2) -- (v3) -- (v4) -- (v5);
\draw (v1) node {$1$};
\draw (v2) node {$0$};
\draw (v3) node {$3$};
\draw (v4) node {$2$};
\draw (v5) node {$1$};
\end{tikzpicture} 
\\
$(a)$ $\bff{r}_0$.
&
$(b)$ Iteration $1$.
&
$(c)$ Iteration $2$.
\\
\\
&
\begin{tikzpicture}[line width=0.5pt, scale=0.95]
\tikzstyle{every node}=[inner sep=0pt, minimum width=12pt] 
\draw (0,0) node (v1) [draw, circle] {};
\draw (1,0) node (v2) [draw, circle] {};
\draw (2,0) node (v3) [draw, circle] {};
\draw (3,0) node (v4) [draw, circle] {};
\draw (4,0) node (v5) [draw, circle] {};
\draw (v1) -- (v2) -- (v3) -- (v4) -- (v5);
\draw (v1) node {$1$};
\draw (v2) node {$4$};
\draw (v3) node {$3$};
\draw (v4) node {$2$};
\draw (v5) node {$1$};
\end{tikzpicture}
& 
\\
&
$(d)$ Iteration $3$.
&
\end{tabular}
\end{center}
\caption{The iterations of Algorithm~\ref{AlgoGen}.}\label{algorithm1}
\end{figure}

Since $|\mathrm{supp}(\bff{r}_0)|=2$, we will have three iterations.
In iteration {\rm 1} we get $W_u(\bff{r})=\{1,5\}$ and $\bff{r}=(1,0,0,2,1)$, see Figure~\ref{algorithm1} (b). 
In iteration {\rm 2}  we get $W=\{1,4\}$ and $\bff{r}=(1,0,3,2,1)$, see Figure~\ref{algorithm1} (c).
In iteration {\rm 3} we get $W=\{1,3\}$ and $\bff{r}=(1,4,3,2,1)$, see Figure~\ref{algorithm1} (d).
Thus,
\[
\bff{r}(P_5,\bff{r}_0,\theta)=(1,4,3,2,1).
\]
Also, if we perform the algorithm with $\bff{r}_0=(1,0,1,0,1)$ and $\theta=(1,2)$, then we get 
\[
\bff{r}(P_5,\bff{r}_0,\theta)=(1,2,1,2,1).
\]
\end{Example}

We are mainly interested in describing the arithmetical structures on a path and on a cycle, in which cases we can give a complete description.
Given $U\subseteq V(G)$, let $\chi_U$ be the characteristic vector of $U$, that is, $(\chi_U)_u=1$ when $u\in U$ and $(\chi_U)_u=0$ when $u\notin U$.

\begin{Corollary}\label{DPnalgo}
If $P_n$ is a path with $n\geq 2$ vertices $v_1,\ldots,v_n$ and $\bff{r}(U, \theta)=\bff{r}(P_n,\chi_U,\theta)$, then
\[
\mathcal{R}(P_n)=\big\{\bff{r}(U, \theta)\, \big|\, \{v_1,v_n\}\subseteq U\subseteq V(P_n)\textrm{ and } \theta \text{ is an order on } V(P_n)-U \big\}.
\]
\end{Corollary}
\begin{proof}
We will use induction on $n$.
If $n=2$, then it is not difficult to see that 
\[
\mathcal{R}(P_2)=\{(1,1)\}=\{\bff{r}(V(P_2),\emptyset)\}.
\]
$(\supseteq)$ Let $\{v_1,v_n\}\subseteq U\subseteq V(P_n)$ and $\theta$ an order on $V(P_n)-U$.
If $U=V(P_n)$, then $\bff{r}(U, \theta)$ is the Laplacian arithmetical structure of $P_n$.
In the other case, let $U'=U$ and $\theta'=\theta|_{V(P_n)-U-v}$ where $v=\theta^{-1}(n-|U|)$.
By the induction hypothesis, $\bff{r}(U',\theta')\in \mathcal{R}(P_{n-1})$. 
Finally, by Corollary~\ref{subdivision} and Algorithm~\ref{AlgoGen}, we get that $\bff{r}(U,\theta)\in \mathcal{R}(P_n)$.

\medskip

$(\subseteq)$ Let $(\bff{d},\bff{r})\in \mathcal{A}(P_n)$.
First, if $(\bff{d},\bff{r})$ is the Laplacian arithmetical structure of $P_n$, then $\bff{r}=\bff{r}(V(P_n),\emptyset)$.
In the other case, by Theorem~\ref{DPnrecursive} there exists $(\bff{d}',\bff{r}')\in \mathcal{R}(P_{n-1})$ such that
$(\bff{d},\bff{r})$ can be obtained from $(\bff{d}',\bff{r}')$ by subdividing and edge $e$.
Let $v$ the vertex of $P_n$ obtained by subdividing $e$.
By the induction hypothesis, $\bff{r}'=\bff{r}(U',\theta')$ for some $\{v_1,v_n\}\subseteq U'\subseteq V(P_{n-1})$ and $\theta$ an order on $V(P_{n-1})-U'$.
Let $U=U'$ and  
\[
\theta(u)=
\begin{cases}
n-|U'| & \text{ if } u=v,\\
\theta'(u) & \text{ otherwise.}
\end{cases}
\]
Finally, by Algorithm~\ref{AlgoGen} and Corollary~\ref{subdivision}, $\bff{r}=\bff{r}(U,\theta)$ and the result follows.
\end{proof}

During the CMO-BIRS Workshop ``Sandpile Groups'' in November 2015, the authors were informed that the 
number of arithmetical structures on the path is a Catalan number.

\begin{Theorem}[\cite{group}]\label{catalan}
The number of arithmetical structures on the path $P_{n+1}$ is equal to the Catalan number 
\[
C_{n}=\frac{1}{n+1}\binom{2n}{n}.
\]
\end{Theorem}
\begin{proof}
This follows directly from~\cite[Problem 93, page 34]{stanley} and the description of the $\bff{r}$ vector of the arithmetical structures on the path 
obtained from the second part of Proposition~\ref{DPnrecursive}.
Note that any sequence obtained in~\cite[Problem 93]{stanley} begins and ends with a $1$.
Therefore, in this case, the operations, appending a $1$ at the end of a sequence and inserting between two consecutive terms their sum, commute.
\end{proof}


\subsection{arithmetical structures on the cycle}\label{ArithCycle}
We begin by giving a recursive description of the arithmetical structures on a cycle.

\begin{Theorem}\label{DCnrecursive}
Let $C_n$ be a cycle with $n\geq 4$ vertices and $(\bff{d},\bff{r})$ be an arithmetical structure of $C_n$.
If ${\bf d}\neq 2\cdot {\bf 1}$, then there exists a vertex $v$ of $C_n$ such that 
\[
{\bf d}_v=1 \text{ and }{\bf d}_u>1\text{ for all }u\in N_{C_n}(v).
\]
Moreover, any arithmetical structure of $C_n$ different from $(\bff{2},\bff{1})$ 
can be obtained from an arithmetical structure of $C_{n-1}$ by subdividing one of its edges.
\end{Theorem}
\begin{proof}
Although the arguments we will give are very similar to those given in the proof of Theorem~\ref{DPnrecursive}, we write down the arguments for the sake of completeness.

Let $C_n=v_1v_2\cdots v_n$.
Since $(2\cdot {\bf 1},\bff{1})$ is an arithmetical structure of $C_n$, by Theorem~\ref{almost}, ${\bf d}\not\geq 2\cdot {\bf 1}$.
Thus, there must exists a vertex $i$ such that ${\bf d}_{v_i}=1$.
Since $n\geq 4$, let $w$ be the vertex adjacent to $v_{i-1}$ different from $v_{i}$.
If ${\bf d}_{v_{i-1}}=1$, then using $L(C_n,{\bf d})\bff{r}^t={\bf 0}^t$ we get that
\[
-\bff{r}_{v_{i-1}}+\bff{r}_{v_i}-\bff{r}_{v_{i+1}}=0\qquad\textrm{and}\qquad -\bff{r}_{w}+\bff{r}_{v_{i-1}}-\bff{r}_{v_i}=0.
\]
Thus $-\bff{r}_{v_{i+1}}-\bff{r}_w=0$.
This is a contradiction to the fact that $\bff{r}>0$.
Similar arguments work for $\bff{d}_{v_{i+1}}$.

Now, let $(\bff{d}',\bff{r}')\in \mathbb{N}_+^{V(C_n)-v_i}\times \mathbb{N}_+^{V(C_n)-v_i}$ be given by $\bff{r}'=\bff{r}|_{V(C_n)-v_i}$
\[
\bff{d}'_u=\begin{cases}
\bff{d}_u-1&\textrm{ if }u=v_{i-1},v_{i+1},\\
{\bf d}_u&\textrm{ otherwise}.
\end{cases}
\]
Since $L(C_n,{\bf d})\bff{r}^t={\bf 0}^t$, then $\bff{r}_{v_i}=\bff{r}_{v_{i-1}}+\bff{r}_{v_{i+1}}$,  
$\bff{r}_{v_{i-1}}\bff{d}_{v_{i-1}}=\bff{r}_{v_{i-2}}+\bff{r}_{v_{i}}$, $\bff{r}_{v_{i+1}}\bff{d}_{v_{i+1}}=\bff{r}_{v_{i}}+\bff{r}_{v_{i+2}}$ 
and therefore
\begin{eqnarray*}
\bff{r}'_{v_{i-1}}\bff{d}'_{v_{i-1}}&=&\bff{r}_{v_{i-1}}(\bff{d}_{v_{i-1}}-1)=\bff{r}_{v_{i-2}}+\bff{r}_{v_{i+1}}=\bff{r}'_{v_{i-2}}+\bff{r}'_{v_{i+1}},\\
\bff{r}'_{v_{i+1}}\bff{d}'_{v_{i+1}}&=&\bff{r}_{v_{i+1}}(\bff{d}_{v_{i+1}}-1)=\bff{r}_{v_{i-1}}+\bff{r}_{v_{i+2}}=\bff{r}'_{v_{i-1}}+\bff{r}'_{v_{i+2}}.
\end{eqnarray*}
Furthermore, since $\bff{d}',\bff{r}'>\bff{0}$, $(\bff{d}',\bff{r}')$ is an arithmetical structure of $C_{n-1}$.
Note that the indices on the $v$'s are taken modulo $n$. 
Finally, the result follows by applying Corollary~\ref{subdivision}. 
\end{proof}

Let $C_2$ be the multigraph with two vertices $v_1$ and $v_2$ and two edges between $v_1$ and $v_2$,
that is, a cycle with two vertices.
It is not difficult to see that $C_2$ has three arithmetical structures, namely $a_1=((2,2),(1,1))$ (the Laplacian), 
$a_2=((4,1),(1,2))$ and $a_3=((1,4),(2,1))$.

In a similar way, $C_3$ has essentially three arithmetical structures, 
namely $b_1=((2,2,2),(1,1,1))$ (the Laplacian), $b_2=((1,3,3),(2,1,1))$ and $b_3=((1,2,5),(3,2,1))$.
Actually $C_3$ has ten arithmetical structures, but the others are permutations of these three.
Also, it is not difficult to see that $b_2$ can be obtained from $a_1$ and $b_3$ can be obtained from $a_2$ by subdividing one of the edges of $C_2$.

In this way, Theorem~\ref{DCnrecursive} can be extended to say that any arithmetical structure of $C_n$ with $n\geq 3$ 
can be obtained by the subdivision of the edges from either the Laplacian arithmetical structures on $C_s$ for some $2\leq s\leq n-1$ or $a_2$.
The next corollary gives a description of $\mathcal{R}(C_n)$ in this sense.

\begin{Corollary}\label{DCnalgo}
If $C_n$ is a cycle with $n\geq 2$ vertices $v_1,\ldots,v_n$ and $\bff{r}(U, \theta)=\bff{r}(C_n,\chi_U,\theta)$, then
\[
\mathcal{R}(C_n)=\big\{\bff{r}(U, \theta)\, \big|\, \emptyset \neq U\subseteq V(C_n)\textrm{ and } \theta \text{ is an order on } V(C_n)-U \big\}.
\]
\end{Corollary}
\begin{proof}
This follows by using similar arguments to those given for Corollary~\ref{DPnalgo}.
\end{proof}

\begin{Example}
Let $C_3$ be the cycle with vertices labeled by $1,2,3$, $\bff{r}_0=(1,0,0)$ and $\theta=(1,2)$.

\begin{figure}[h]
\begin{center}
\begin{tabular}{c@{\hskip 15mm}c@{\hskip 15mm}c}
\begin{tikzpicture}[line width=0.5pt, scale=0.95]
\tikzstyle{every node}=[inner sep=0pt, minimum width=12pt] 
\draw (-30:1) node (v1) [draw, circle] {};
\draw (90:1) node (v2) [draw, circle] {};
\draw (210:1) node (v3) [draw, circle] {};
\draw (v1) -- (v2) -- (v3) -- (v1);
\draw (v1) node {$1$};
\draw (v2) node {$0$};
\draw (v3) node {$0$};
\end{tikzpicture}
&
\begin{tikzpicture}[line width=0.5pt, scale=0.95]
\tikzstyle{every node}=[inner sep=0pt, minimum width=12pt] 
\draw (-30:1) node (v1) [draw, circle] {};
\draw (90:1) node (v2) [draw, circle] {};
\draw (210:1) node (v3) [draw, circle] {};
\draw (v1) -- (v2) -- (v3) -- (v1);
\draw (v1) node {$1$};
\draw (v2) node {$2$};
\draw (v3) node {$0$};
\end{tikzpicture}
&
\begin{tikzpicture}[line width=0.5pt, scale=0.95]
\tikzstyle{every node}=[inner sep=0pt, minimum width=12pt] 
\draw (-30:1) node (v1) [draw, circle] {};
\draw (90:1) node (v2) [draw, circle] {};
\draw (210:1) node (v3) [draw, circle] {};
\draw (v1) -- (v2) -- (v3) -- (v1);
\draw (v1) node {$1$};
\draw (v2) node {$2$};
\draw (v3) node {$3$};
\end{tikzpicture} 
\\
$(a)$ $\bff{r}_0$.
&
$(b)$ Iteration $1$.
&
$(c)$ Iteration $2$.
\end{tabular}
\end{center}
\caption{The iterations of Algorithm~\ref{AlgoGen} on $C_3$ with $\bff{r}_0=(1,0,0)$.}\label{algorithm2}
\end{figure}
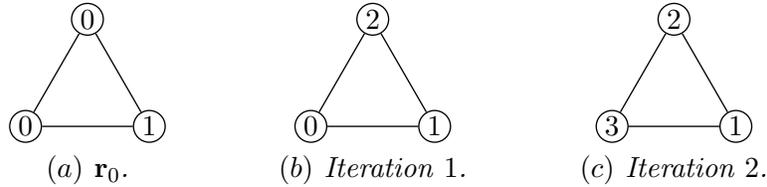
In iteration {\rm 1} we get $W_u(\bff{r})=\{1,1\}$ and $\bff{r}=(1,2,0)$, see Figure~\ref{algorithm2} (b).
In iteration {\rm 2}  we get $W=\{1,2\}$ and $\bff{r}=(1,2,3)$, see Figure~\ref{algorithm2} (c).
Thus, $\bff{r}(C_3,\bff{r}_0,\theta)=(1,2,3)$.
\end{Example}

Another way to see the arithmetical structures on $C_n$ is as the arithmetical structures 
obtained from $P_{n+1}$ by merging the terminal vertices of $P_{n+1}$.
More precisely, let $P_{n+1}=u_1u_2\cdots v_{n+1}$ and $C_n=v_1v_2\cdots v_{n}$.
If $(\bff{d},\bff{r})\in \mathcal{A}(P_{n+1})$, then by Corollary~\ref{DCnalgo}, $\bff{r}_{u_1}=\bff{r}_{u_{n+1}}=1$.
Making a proper identification of the vertices of $P_{n+1}$ and $C_n$, Theorem~\ref{merge} implies that for any $1\leq k\leq n$ 
there exists a bijection between $\mathcal{A}(P_{n+1})$ and the arithmetical structures on $C_n$ with $\bff{r}_{v_k}=1$.
Note that this bijection can also be obtained in the following way: 
let $f_k :V(P_{n+1})\rightarrow V(C_n)$ be given by $f_k(u_i)=v_{n-k+i+1\, (\textrm{mod } n)}$ and $f_k(u_{n+1})=v_k$.
Let $\tilde{f}_k$ be the map between $\mathcal{A}(P_{n+1})$ and $\mathcal{A}_k(C_n)=\{(\bff{d},\bff{r})\in \mathcal{A}(C_{n}) \, | \, \bff{r}_{v_k}=1\}$
induced by $f_k$ on the right sides of Corollaries~\ref{DPnalgo} and~\ref{DCnalgo}.
That is,
\[
\tilde{f}_k(\bff{r}(P_{n+1},U,\theta))=\bff{r}(C_n,f(U),\theta\circ f^{-1}).
\]
Using this correspondence, we get that 
\[
\frac{2}{n+1}\binom{2n-1}{n-1}=C_n \leq |\mathcal{A}(C_{n})|\leq nC_n=\frac{2n}{n+1}\binom{2n-1}{n-1}.
\]
In a similar way, using Theorem~\ref{split} we have a correspondence between $\mathcal{A}(C_{n})$ and $\mathcal{A}(P_{n+1})$.
This correspondence can also be obtained with a mapping between the vertices of $C_n$ and $P_{n+1}$.

During the BIRS-CMO workshop ``Sandpile groups'' the authors were informed that the number of 
arithmetical structures on a cycle with $n$ vertices is equal to $\binom{2n-1}{n-1}$, see~\cite{group}.


\subsection{arithmetical structures on the subdivision of a graph}\label{ArithSG}
Given an arithmetical structure of a graph $G$, Algorithm~\ref{AlgoGen} produces arithmetical structures on any subdivision of $G$.
More precisely, let $s(G)$ the graph obtained from $G$ by subdividing several edges several times and $(\bff{d},\bff{r})$ an arithmetical structure of a graph $G$.
Taking $s(G)$, $\theta$ an order in $V(s(G))-V(G)$, and
\[
(\bff{r}_0)_u=
\begin{cases}
\bff{r}_u & \text{ if } u\in V(G),\\
0 & \text{otherwise},
\end{cases}
\] 
as the input of Algorithm~\ref{AlgoGen}, we can generate some of the arithmetical structures on $s(G)$.

For instance, consider the complete graph $K_4$ with four vertices and $s(K_4)$ the graph obtained from $K_4$ by subdividing the edge $v_3v_4$.
It is not difficult to check that $K_4$ has $191$ arithmetical structures divided into $12$ classes.
One of them is given by $\bff{d}=(1,2,9,14)$ and $\bff{r}=(15,10,3,2)$.
Using Algorithm~\ref{AlgoGen} with $\bff{r}_0=(15,10,3,2,0)$ we get that 
\[
\bff{d}=(1,2,10,15,1)\text{ and }\bff{r}=(15,10,3,2,5)
\] 
is an arithmetical structure of $s(K_4)$.
In a similar way with $\bff{r}_0=(2,3,3,5,0)$ we get that $\bff{d}=(2,3,4,6,1)$ and $\bff{r}=(4,3,3,2,5)$ is an arithmetical structure of $s(K_4)$.
Moreover, if we subdivide twice the edge $v_3,v_4$, we can get more arithmetical structures, for instance with $\bff{r}_0=(15,10,3,2,0)$
we get two new arithmetical structures 
\[
\bff{d}=(1,2,10,16,2,1), \bff{r}=(15,10,3,2,7) \text{ and }\bff{d}=(1,2,11,15,1,2), \bff{r}=(15,10,3,2,8,5).
\]
Also, beginning with $\bff{r}_0=(1,1,1,1,1,0)$ we get the arithmetical structure $\bff{d}=(3,3,3,4,3,1), \bff{r}=(1,1,1,1,1,2)$.

This procedure can be done in general and therefore we can get a lower bound for the arithmetical structures on any subdivision of a graph.
\begin{Proposition}
Let $G$ be a graph and let $s(G)$ be the graph obtained from $G$ by subdividing $n_e$ the edge $e$ of $G$.
Then
\[
|\mathcal{A}(s(G))|\geq (|\mathcal{A}(G)|-1)\cdot \prod_{e\in E(G)} C_{n_e} + \prod_{e\in E(G)} C_{n_e+1}.
\]
\end{Proposition}
\begin{proof}
The first term follows by Algorithm~\ref{AlgoGen} using $\bff{1}\neq \bff{r}_0\in \mathcal{R}(G)$ because there are $C_{n_e} $ 
orders in each of the paths obtained by subdividing $n_e$ times the edge $e$.
The second term follows by Algorithm~\ref{AlgoGen} using as $\bff{r}_0$ the $\bff{r}$ of the Laplacian arithmetical structure of the graph obtained 
by subdividing $n'_e$ the edge $e$ of $G$ for all $n'_e\leq n_e$ and $e\in E(G)$ and Theorem~\ref{catalan}.
\end{proof}

Unlike paths and cycles which are subdivisions of $K_2$ and $K_3$ respectively, 
in general we can not get all the arithmetical structures on a subdivision of a graph $G$ using Algorithm~\ref{AlgoGen}.
The next example shows this for the subdivision of $K_4$.

\begin{Example}
Consider the graph $s(K_4)$ obtained by subdividing the edge $v_3v_4$ of the complete graph with four vertices $K_4$.
\[
\left(
\begin{array}{ccccc}
4 & -1 & -1 & -1 & 0\\
-1 & 4 & -1 & -1 & 0\\
-1 & -1 & 2 & 0 & -1\\
-1 & -1 & 0 & 2 & -1\\
0 & 0 & -1 & -1 & 3
\end{array}
\right)
\left(
\begin{array}{c}
2\\
2\\
3\\
3\\
2
\end{array}
\right)
=\bff{0}
\]
If $\bff{d}=(4,4,2,2,3)$ and $\bff{r}=(2,2,3,3,2)$, then $(\bff{d},\bff{r})$ is an arithmetical structure of $s(K_4)$ 
which can not be obtained by using Corollary~\ref{subdivision} or Algorithm~\ref{AlgoGen}.
\end{Example}

We end with posing a conjecture about the number of arithmetical structures on a graph.

\begin{Conjecture}\label{conj}
If $H$ is a simple connected graph with $n$ vertices, then 
\[
|\mathcal{A}(P_n)| \leq |\mathcal{A}(H)|\leq |\mathcal{A}(K_n)|.
\]
\end{Conjecture}

The inequality on the right comes from a private communication with Lionel Levine, who asked us whether the complete graph is the connected graph with the largest  number of arithmetical structures.

\noindent {\bf Acknowledgments}
Some of the results contained in this article were presented at the 2015 CMO-BIRS Workshop 15w5119 on sandpile groups.
The authors would like to thank the CMO-BIRS Institute for their support and hospitality, and all the people involved in the workshop for creating a vibrant research community. 
In particular, the authors would like to thank Benjamin Braun, Scott Corry, Art Duval, Luis Garcia, Darren Glass, Nathan Kaplan, 
Lionel Levine, Jeremy Martin, Criel Merino, and Gregg Musiker for their helpful comments.

The authors would like to thank the anonymous referee for his helpful comments.


\end{document}